\documentclass[a4paper, reqno, 11pt]{amsart}
\makeatletter

\@addtoreset{equation}{section}
\makeatother
\usepackage{setspace}
\usepackage{xcolor}
\usepackage{amsmath}
\usepackage{amssymb}
\usepackage{latexsym}
\usepackage{amsthm}
\usepackage[dvipdfm]{hyperref}
\newtheorem{Thm}{Theorem}[section]

\newtheorem{Lem}[Thm]{Lemma}
\newtheorem{Prop}[Thm]{Proposition}

\newtheorem{Cor}[Thm]{Corollary}

\theoremstyle{definition}

\newtheorem{Rem}[Thm]{Remark}
\newtheorem{Exa}[Thm]{Example}

\newcommand{\N}{\mathbb{N}} 
\newcommand{\Z}{\mathbb{Z}} 
\newcommand{\R}{\mathbb{R}}
\newcommand{\F}{\mathcal{F}}

\begin{document}

\title[]{Hausdorff dimensions for graph-directed measures driven by infinite rooted trees} 
\author{Kazuki Okamura}
\address{School of General Education, Shinshu University, 3-1-1, Asahi, Matsumoto, Nagano, 390-8621, JAPAN.} 
\email{kazukio@shinshu-u.ac.jp} 

\begin{abstract}
We give upper and lower bounds for the Hausdorff dimensions for a class of graph-directed measures when its underlying directed graph is the {\it infinite} $N$-ary tree. 
These measures are different from graph-directed self-similar measures driven by finite directed graphs and are not necessarily Gibbs measures. 
However our class contains several measures appearing in fractal geometry and functional equations, specifically, 
measures defined by restrictions of non-constant harmonic functions on the two-dimensional Sierp\'inski  gasket, 
the Kusuoka energy measures on it, and, measures defined by solutions of de Rham's functional equations driven by linear fractional transformations. 
\end{abstract}

\maketitle

\setcounter{tocdepth}{1}
\tableofcontents 

\section{Introduction and Main results}

Let $N \ge 2$. 
Let $\Sigma_N := \{0,1, \dots, N-1\}$. 
Let $Y$ be a set and $Z$ be a measurable space. 
For $i \in \Sigma_N$, 
let $G_i : Y \to [0,1]$ and $H_i : Y \to Y$ be maps and $F_i : Z \to Z$ be measurable maps.   
Then we consider the following equation for a family of probability measures $\{\mu_y\}_{y \in Y}$ on $Z$:  
\begin{equation}\label{g-me} 
\mu_{y} = \sum_{i \in \Sigma_N} G_i (y) \mu_{H_i (y)} \circ F_i^{-1}, 
\end{equation} 
under the assumption that 
\begin{equation}\label{G-sum-1}
\sum_{i \in \Sigma_N}  G_i (y) = 1. 
\end{equation} 

Equations for probability measures similar to \eqref{g-me} appear in the Markov type measures in  Edgar-Mauldin \cite{EM92}, 
a self-similar family of measures in Strichartz \cite[Definition 2.2]{St93} and the graph directed self-similar measures in Olsen \cite[Section 1.1]{Ol94}.  
Let $(V, E)$ be a directed graph where multi-edges and self-loops are allowed. 
\cite{EM92}, \cite{St93} and \cite{Ol94} focus on the case that $V$ is {\it finite}. 
Several arguments in these references such as the Perron-Frobenius theorem and ergodic theorem for finite Markov chains depend on the fact that $V$ is finite. 
Here we consider the case that $(V, E)$ is the {\it infinite} $N$-ary tree on which each edge is equipped with a direction from a vertex closer to the root to its descendants, and that $\{F_i\}_i$ is an iterated function system and $Z$ is its attractor. 

If $V = Y(y) = \{y\}$ (See \eqref{Def-Yy} and Remark \ref{main-rem} (ii)  below for the definition of $Y(y)$.), 
then,  
\[ \mu_{y} = \sum_{i \in \Sigma_N} G_i (y) \mu_{y} \circ F_i^{-1}. \]
If $\{F_i\}_i$ is a family of contractions on a complete metric space, then, $\mu_y$ is the invariant measure of the iterated function system $(Z, \{F_i\}_{i \in \Sigma_N})$ equipped with probability weights $(G_i(y))_{i \in \Sigma_N}$. 
However, this paper rather focuses on the case that $\mu_y$ is {\it not} an invariant measure of an iterated function system. 
Our class contains several measures which are not not necessarily invariant or Gibbs measures and appear in fractals and functional equations, 
specifically, measures defined by restrictions of non-constant harmonic functions on the two-dimensional standard Sierp\'inski  gasket, the Kusuoka energy measures on it, and furthermore measures defined by solutions of de Rham's functional equations driven by linear fractional transformations. 

We give upper and lower bounds for the Hausdorff dimensions for $\{\mu_y\}_{y \in Y}$ 
under certain regularity conditions for the functions $\{G_i\}_{i \in \Sigma_N}$ and $\{H_i\}_{i \in \Sigma_N}$ on $Y$ 
and the assumption that $\{F_i\}_{i \in \Sigma_N}$ is an iterated function system  on a complete metric space satisfying certain conditions. 
We now state some applications of our results. 
We extend the main result of \cite{ADF14} considering restrictions of non-constant harmonic functions on the two-dimensional standard Sierp\'inski  gasket. 
Our proof gives an alternative proof of singularity of the Kusuoka energy measures on the standard 2-dimensional Sierp\'inski  gasket \cite{Ku89}.   
Furthermore, de Rham's functional equations driven by linear fractional transformations considered in \cite{Ok14} are also generalized. 
Specifically, we deal with the case that an equation is driven by $N$ linear fractional transformations which are {\it weak} contractions. \cite{Ok14} deals with the case that an equation is driven by only two linear fractional transformations which are contractions.
We also discuss singularity of $\mu_y$ with respect to self-similar measures of iterated function systems equipped with probability weights which are not a canonical measure on $Z$. 

\subsection{Framework and main results}

Now we describe our framework and main results. 

Let $\Sigma_N := \{0,1, \dots, N-1\}, \ N \ge 2$. 
Let $S_i, i \in \Sigma_N$, be contractive maps on a complete metric space $(M, d)$, and $K$ be the attractor of the iterated function system $\{S_i\}_{i \in \Sigma_N}$.  
We put the Borel $\sigma$-algebra on $K$ induced by the metric $d$ on $M$.    
For $x \in K$ and $r > 0$, let $B(x,r) = \{y \in K : d(x,y) \le r\}$. 
For $m \ge 1$, let $S_{i_1 \cdots i_m} := S_{i_1} \circ \cdots \circ S_{i_m} $ and $K_{i_1 \cdots i_m} := S_{i_1 \cdots i_m}(K)$.  
For each $i \in \Sigma_N$, let $r_i$ be the Lipschitz constant of $S_i$, in other words, 
\[ r_i := \sup_{x,y; x \ne y} \frac{d(S_i(x), S_i(y))}{d(x,y)}.  \]
Assume that $0 < r_i < 1$ for each $i \in \Sigma_N$. 
Let $s$ be a positive real number such that 
\begin{equation}\label{s} 
\sum_{i \in \Sigma_N} r_i^{s} = 1. 
\end{equation}

Let $Y$ be a set. 
Let $H_i : Y \to Y, \ i \in \Sigma_N$, be maps. 
For each $y \in Y$, let 
\begin{equation}\label{Def-Yy} 
Y(y) := \left\{ H_{i_{l}} \circ \cdots \circ H_{i_{1}}(y) \in Y \mid i_1, \dots, i_l \in \Sigma_N, \ l \ge 1 \right\}. 
\end{equation} 
Let $G_i : Y \to [0,1], \ i \in \Sigma_N$, be maps satisfying that for each $y \in Y$, $\displaystyle \sum_{i \in \Sigma_N}  G_i (y) = 1.$
We remark that this sum is taken with respect to $i \in \Sigma_N$, not to $y \in Y$. 

Let $\dim_H A$ be the Hausdorff dimension of $A \subset K$. 
Let the Hausdorff dimension of a measure $\mu$ be $\dim_H \mu := \inf\{\dim_H B \mid \mu(B) = 1\}$.
 
The following is our main result.  

\begin{Thm}\label{main} 
Under the above setup, a family of probability measures $\{\mu_y\}_{y \in Y}$ on the attractor $K$  satisfying that 
\begin{equation*}\label{g-me-K} 
\mu_{y} = \sum_{i \in \Sigma_N} G_i (y) \mu_{H_i (y)} \circ S_i^{-1}, \ y \in Y, 
\end{equation*} 
exists and is unique. 
Furthermore we have the following two statements for every $y \in Y$: \\ 
(i) Let $s$ be the positive number in \eqref{s}. 
Assume that there exist a constant $\displaystyle \epsilon_0 \in \left(0, \left(\min_{i \in \Sigma_N} r_i\right)^s\right)$ and an integer $l \ge 1$ such that 
\begin{align}\label{multi-sep}
Y(y) &\cap \bigcap_{i \in \Sigma_N} G_{i}^{-1}\left( \left[ r_i^{s} - \epsilon_0, r_i^{s} + \epsilon_0\right]\right) \notag\\
&\cap \bigcap_{j \in \Sigma_N} \left(G_{j} \circ H_{i_{l}} \circ \cdots \circ H_{i_{1}}\right)^{-1} \left( \left[ r_j^{s} - \epsilon_0, r_j^{s} + \epsilon_0\right]\right) = \emptyset 
\end{align}
 for every $i_1, \dots, i_{l} \in  \Sigma_N$. 
Then,  
\begin{equation}\label{strict} 
\dim_H \mu_y < s. 
\end{equation}
(ii) Assume that there exists constants $r \in (0,1), c > 0$ and $D > 0$ such that for every $x \in K$ and $m \ge 1$, 
$$\left| \left\{ (i_1, \cdots, i_m) \in (\Sigma_N)^{m}  \ \middle| \   B(x, c r^m) \cap  K_{i_1 \cdots i_m} \ne \emptyset \right\} \right| \le D,$$  
and furthermore it holds that 
\begin{equation}\label{wAy} 
c \le  \inf_{w \in Y(y)}  G_i (w) \le \sup_{w \in Y(y)}  G_i (w) \le 1-c 
\end{equation}
 for some $i \in \Sigma_N$ and $c > 0$. 
Then, $$\dim_H \mu_y > 0.$$  
\end{Thm} 

\begin{Rem}\label{main-rem} 
(i) If $(M,d)$ is a Euclid space and $S_i, i \in \Sigma_N,$ are similitudes and satisfy the open set condition, then, it holds that $\dim_H K = s$ for the attractor $K$ and the positive number $s$ in \eqref{s} (\cite[Theorem 9.3]{F14}). \\
(ii) The subset $Y(y)$ of $Y$ in \eqref{Def-Yy} corresponds to the infinite $N$-ary tree with root $y$, specifically, the set of vertices is $Y(y)$, and the set of directed edges is given by 
$$\left\{ \left(H_{i_{l}} \circ \cdots \circ H_{i_{1}}(y), \ H_{i_{l+1}} \circ H_{i_{l}} \circ \cdots \circ H_{i_{1}}(y)\right) \mid i_1, \dots, i_l, i_{l+1} \in \Sigma_N, \ l \ge 1 \right\},$$
where we let $H_{i_{l}} \circ \cdots \circ H_{i_{1}}(y) := y$ if $l = 0$. \\
(iii) In the setup above, we regard $Y$ simply as a set, but we will often put a metric structure and a linear structure on $Y$.   
Features of our setup are that $Y$ is only a set and irrelevant to $K$,  $Y$ contains countably many points, $G_i$ and $H_i$, $i \in \Sigma_N$, are not constant functions, and furthermore, $\mu_y$ is {\it not} an invariant measure of a certain iterated function system equipped with probability weights. \\
(iv) The assumption for $(K, \{S_i\}_i)$ in Theorem \ref{main} (ii) is the same as in assumption (2)  in \cite[Corollary 1.3]{Kig95}. 
It is known by Bandt and Graf \cite{BG92} that if this holds then the open set condition for $\{S_i\}_i$ holds. 
It is satisfied if $(M,d)$ is the Euclid space and $S_i, i \in \Sigma_N,$ are similitudes and satisfy the open set condition (See the proof of \cite[Theorem 9.3]{F14}). 
In examples which are dealt with later, \eqref{strict} implies that $\mu_y$ is singular with respect to a ``canonical" probability measure on the attractor $K$ whose Hausdorff dimension is the positive number $s$ satisfying \eqref{s}.  
In the case of homogeneous self-similar sets, it holds that $s = \log N/\log (1/r)$. 
\end{Rem}

As an outline of our proof, we follow \cite[Theorem 1.2]{Ok14}. 
However, our method is more transparent than \cite{Ok14}.
Specifically, we do not use four kinds of random subsets of natural numbers as in \cite[Lemma 3.3]{Ok14}.      
See Lemma \ref{key} below for an alternative way.   
We emphasize that not only Theorem \ref{main} is applicable to the models described above, but also there is potential for applications to different models. 
Indeed, we deal with some special examples other than the models described above.  

\subsection{Comparison with related results}

Our purpose is to know whether $\mu_y$ is absolutely continuous or singular with respect to a natural canonical probability measure on $Z$.   
In several examples, we deal with the case that $Z = K = [0,1]$. 
If $\mu_y$ is singular, then, the distribution function of $\mu_y$ has derivative zero at almost every point. 
The singularity problem proposed by Kaimanovich \cite{Kai03} is a little similar to our motivation. 
\cite{Kai03} considers whether the harmonic measure on a boundary of Markov chain is singular with respect to a ``canonical" measure on the boundary.  
It is also interesting to consider whether not only the harmonic measure but also other measures on the boundary defined in natural ways are singular or not with respect to the canonical measure. 
Indeed, in \cite[p.180]{Kai03}, investigating dynamical properties for the Kusuoka energy measure is proposed, and recently \cite{JOP17} considers them. 
The Kusuoka energy measure is {\it not} a Gibbs measure, and therefore, techniques of the thermodynamic formalism are not suitable to apply. 
See \cite[Section 3]{JOP17} for more details. 

If the Hausdorff dimension of $\mu_y$ is strictly smaller than the Hausdorff dimension of a canonical measure, 
then, $\mu_y$ is singular with respect to the canonical measure. 
So, it is desirable to know an exact value or a good upper bound of the Hausdorff dimension of $\mu_y$.  
In general, for iterated function systems with place-dependent probability weights, deriving a dimension formula for invariant measures is valuable. 
For a class of iterated function systems which are driven by non-linear weak contractions and equipped with place-dependent probability weights, 
Jaroszewska-Rams \cite{JR08} gave an upper bound for the Hausdorff dimension of an invariant measure in the form of the entropy divided by the Lyapunov exponent. 
For a large class of iterated function systems driven by similitudes with considerable overlaps, Hochman \cite{Ho14} obtained the exact dimension formula.  
In order to give an upper bound for the Hausdorff dimension, 
we need to estimate the entropy and the Lyapunov exponent, both of which are the values of integrands with respect to $\mu_y$, which might be {\it singular} with respect to the canonical measure on $K$.
However, intricate calculations are actually required for estimating the Hausdorff dimensions of invariant measures of an iterated function system, 
 in particular when the probability weights of the iterated function system are place-dependent or the iterated function system is driven by non-similitudes.  
See B\'ar\'any-Pollicott-Simon \cite[Sections 7 and 8]{BPS12} and B\'ar\'any \cite[Section 5]{Ba15} for example.   

In the case that $Y$ consists of only one point and furthermore all $F_i, i \in \Sigma_N$, are similitudes on a complete metric space satisfying assumption (2)  in \cite[Corollary 1.3]{Kig95}, then, by \cite{BG92}, 
the open set condition holds for $\{F_i\}_i$ and $\mu_y$ is a self-similar measure where the corresponding probabilities $\{G_i\}_i$ are not place-dependent. 
However, we mainly focus on the case that $Y$ contains at least {\it countably many} points and the case that $\mu_y$ might not be a Gibbs measure or an invariant measure of an iterated function system.  

In the case that $Y$ contains countably many points and has a linear structure, the values of $G_i (y)$ and $H_i(y)$, $i \in \Sigma_N$, can be {\it non-linear} functions on $Y$. 
So, even if we have a form of dimension formula for $\mu_y$ expressed by $\{G_i\}_i$, $\{H_i\}_i$,  and $\mu_y$,  
it is difficult for estimating $\dim_H \mu_y$ from possible dimension formulae. 
In this paper, we focus on the issue whether $\mu_y$ is singular or not with respect to the canonical measure on $K$ whose Hausdorff dimension is the positive number $s$ satisfying \eqref{s} above, 
rather than pursuing a form of dimension formula for $\mu_y$. 
We give upper and lower bounds for the Hausdorff dimensions of $\mu_y$ without deriving any forms of dimension formulae. 

\vspace{1\baselineskip} 

The rest of the paper is organized as follows. 
Section 2 is devoted to the proof of Theorem \ref{main}. 
Section 3 is devoted to deal with various examples including the restriction of harmonic functions, the energy measures on the Sierp\'inski gasket and de Rham's functional equations. 
In Section 4, we state open problems. 

\section{Proof of Theorem \ref{main}}  

Let $\N := \{1,2, \dots \}$. 
Henceforth we put the product $\sigma$-algebra on a symbolic space $(\Sigma_N)^{\N}$.  
Let $\pi : (\Sigma_N)^{\N} \to K$ be a surjective map such that 
\begin{equation}\label{commute}
\pi(ix) = S_i (\pi(x)) \textup{ for every $i \in \Sigma_N$ and $x \in (\Sigma_{N})^{\N}$}.  
\end{equation}   
$\pi$ is uniquely determined and is called the natural projection.   

For $n \in \N$, 
let $X_n (x)$ be the projection of $x \in (\Sigma_N)^{\N}$ to $n$-th coordinate.  
Let $\F_n := \sigma(X_1, \dots, X_n)$.
This is a $\sigma$-algebra on $(\Sigma_N)^{\N}$.  
For $n \ge 1$, let a cylinder set 
\[ I(i_1, \dots, i_n) :=  \left\{ w \in (\Sigma_N)^{\N} \ \middle| \   X_k (w) = i_k, 1 \le k \le n \right\}, \ i_1, \dots, i_n \in \Sigma_N. \]
Consider \eqref{g-me} in the case that $Z = (\Sigma_N)^{\N}$ and $F_i (x) = ix$. 

By \eqref{G-sum-1} and the Kolmogorov extension theorem,   
there exists a unique probability measure $\nu_y$ on $(\Sigma_N)^{\N}$ such that 
\[ \nu_y \left( I(i_1, \dots, i_n) \right) = \prod_{k = 1}^{n} G_{i_{k}} \circ H_{i_{k-1}} \circ \cdots \circ H_{i_{1}}(y). \]
Then, $\{\nu_y\}_{y \in Y}$ is a family of solutions of  \eqref{g-me}. 
Then, by \eqref{commute}, a family of the push-forward measures $\{\nu_y \circ \pi^{-1}\}_{y \in Y}$ is a family of solutions of \eqref{g-me} for the case that $Z = K$ and $F_i = S_i$, $i \in \Sigma_N$.  

Since each $S_i$ is contractive, 
then, by following the proof of \cite[Theorem 2.8]{F97},  
we can show that a family of solutions of \eqref{g-me} for $Z = K$ and $F_i = S_i$ is unique.  
It holds that for every $y \in Y$, 
\begin{equation}\label{mu-nu-pi}
\nu_y \circ \pi^{-1} = \mu_y. 
\end{equation}

Recall the definition of the positive number $s$ in \eqref{s}. 
For every $y \in Y$, $x \in (\Sigma_N)^{\mathbb N}$ and $n \in \mathbb{N}$, let 
\[ R_{y,n} (x) := \frac{\nu_y \left( I(X_1 (x), \dots, X_n (x)) \right)}{(r_{X_1 (x)} \cdots r_{X_n (x)})^s}.\] 
Then it follows from induction in $n$ that 
\begin{Lem}\label{Lem-1}
For every $y \in Y$, $x \in (\Sigma_N)^{\mathbb N}$ and $n  \in \mathbb{N}$, it holds that 
\[ R_{y, n+1} (x) = R_{y, n} (x) G_{X_{n+1}(x)} \circ H_{X_{n}(x)} \circ \cdots \circ H_{X_{1}(x)}(y) (r_{X_{n+1} (x)})^s. \]  
\end{Lem} 

Let 
\[ P_N := \left\{(p_0, \dots, p_{N-1}) \in [0,1]^N  \mid   \sum_{i \in \Sigma_N} p_i = 1 \right\}.\]   
Define a relative entropy $s_N : P_N \to \R$ with respect to $(r_i^s)_{i \in \Sigma_N}$ by 
\[ s_N \left(p_0, \dots, p_{N-1}\right) := \sum_{i \in \Sigma_N} p_i \log \frac{r_i^s}{p_i}, \]
where we put $0 \log 0  =0$.  
We remark that $s_N \left(p_0, \dots, p_{N-1}\right) \le 0$ and $s_N \left(p_0, \dots, p_{N-1}\right) = 0$ if and only if $\left(p_0, \dots, p_{N-1}\right) = \left(r_0^s, \dots, r_{N-1}^s\right)$.  

It can occur that $R_{y, n-1}(x) = 0$, however, it holds that 
$$\nu_y (\{x \in (\Sigma_N)^{\N} \mid R_{y, n-1}(x) = 0\}) = 0$$  
for every $n \ge 1$.  
Hence if we say ``$\nu_y$-a.s.$x$", then, we can assume that $R_{y, n-1}(x) > 0$ for every $n$.  
Then, by Lemma \ref{Lem-1},  
\begin{Lem}\label{ent-condi}
For every $y \in Y$, it holds that 
\[ E^{\nu_y}\left[ -\log \left(\frac{R_{y,n}}{R_{y,n-1}}\right) \ \middle| \  \F_{n-1} \right] (x) = s_N \left( \left(G_{j} \circ H_{X_{n-1}(x)} \circ \cdots \circ H_{X_{1}(x)}(y)\right)_{j \in \Sigma_N} \right), \textup{ $\nu_y$-a.s.$x$. }\] 
\end{Lem} 

Let $M_{y,0} = 0$. 
For $n \ge 1$, let 
\[ M_{y,n} - M_{y, n-1} := -\log \left(\frac{R_{y,n}}{R_{y, n-1}}\right) - E^{\nu_y}\left[ -\log\left(\frac{R_{y,n}}{R_{y,n-1}}\right) \ \middle| \  \F_{n-1} \right]. \]
(If $R_{y, n-1}(x) = 0$, then, we let 
$(M_{y,n} - M_{y, n-1}) (x) := 0$.  
but such $x$ does not affect integrations with respect to $\nu_y$.)   

Then, $\{M_{y,n}, \mathcal{F}_n\}_{n \ge 0}$ is a martingale under $\nu_y$ and we have that 

\begin{Lem}\label{M-to-0}
For every $y \in Y$, it holds that 
\[ \lim_{n \to \infty} \frac{M_{y,n}}{n} = 0, \ \nu_y\textup{-a.s.} \]
\end{Lem} 

\begin{proof} 
This part will be shown in the same manner as \cite[Lemma 2.3 (2)]{Ok14} by using Jensen's inequality and Doob's submartingale inequality.   
Let 
\[ C := \sup_{(p_i)_i \in P_N} \sum_{i \in \Sigma_N} p_i (-\log p_i)^2 < +\infty.\]  
Then, for every $n \ge 1$,  
\[ E^{\nu_y} \left[ \left( \log \frac{R_{y, n+1}}{R_{y, n}}  \right)^2\right] \le \frac{C}{(\min_i r_i)^{2s}}. \]

By this and Jensen's inequality, 
\[ \sup_{n \ge 0} E^{\nu_y} \left[ \left(M_{y, n+1} - M_{y, n}\right)^2 \right] \le \frac{4C}{(\min_i r_i)^{2s}}.\] 

By Doob's submartingale inequality, 
we have that for every $\epsilon > 0$ and every $n \ge 1$, 
\begin{align*} 
\nu_y \left( \max_{1 \le k \le 2^n} |M_{y, k}| \ge \epsilon 4^n \right) &\le \frac{E^{\nu_y} \left[ (M_{y, 2^n})^2\right]}{\epsilon 4^n} \\
&= \frac{ \sum_{k \le 2^n} E^{\nu_y} \left[ \left(M_{y, k} - M_{y, k-1}\right)^2 \right]  }{\epsilon 4^n} \le \frac{C}{(\min_i r_i)^{2s} \epsilon 4^{n-1}}.  
\end{align*} 

Therefore we have 
\[ \limsup_{n \to \infty} \frac{\left|M_{y, n}\right|}{n} \le \sqrt{\epsilon}, \textup{ $\nu_y$-a.s.}\]
\end{proof} 

For $i \ge 1$, $y \in Y$ and $x \in (\Sigma_N)^{\N}$, 
let 
\[ h_i (y; x) := H_{X_{i}(x)} \circ \cdots \circ H_{X_{1}(x)}(y),\]  
and 
\[ p_{i} (y; x) := \left( G_{j} \circ  h_i (y,x) \right)_{j \in \Sigma_N} =  \left( G_{j} \circ H_{X_{i}(x)} \circ \cdots \circ H_{X_{1}(x)}(y) \right)_{j \in \Sigma_N}.  \] 

By Lemmas \ref{ent-condi} and \ref{M-to-0},  we have that  
\begin{Prop}\label{transfer} 
For every $y \in Y$, it holds that 
\[ \limsup_{n \to \infty} \frac{-\log R_{y,n} (x)}{n} = \limsup_{n \to \infty} \frac{1}{n} \sum_{i=1}^{n} s_N \left( p_{i-1} (y; x) \right), \textup{ $\nu_y$-a.s.$x$} \] 
and, 
\[ \liminf_{n \to \infty} \frac{-\log R_{y,n} (x)}{n} = \liminf_{n \to \infty} \frac{1}{n} \sum_{i=1}^{n} s_N \left( p_{i-1} (y; x) \right), \textup{ $\nu_y$-a.s.$x$}. \] 
\end{Prop} 

\subsection{Proof of (i)}

Throughout this subsection, we assume the assumption of Theorem \ref{main} (i). 

\begin{Lem}\label{exist}
Let $y \in Y$. 
If for some positive constant $\epsilon_1$, 
\[ \limsup_{n \to \infty} \frac{-\log R_{y,n}}{n} \le -\epsilon_1, \textup{ $\nu_y$-a.s.,} \]
then, 
there exists a Borel subset $B_0 \subset K$ such that 
$\mu_y (B_0) = 1$ and 
\[ \dim_H (B_0) < s. \]  
\end{Lem}

\begin{proof}
For $0 < t < \epsilon_1$, let 
$$A_{n, t} := \left\{-\log R_n \le n(-\epsilon_1 + t) \right\} \subset (\Sigma_N)^{\N}.$$
Then, by the assumption  and \eqref{mu-nu-pi}, 
\[ \mu_y \left( \pi \left(\bigcap_{k \ge 1} \bigcup_{m \ge 1} \bigcap_{n \ge m}  A_{n, 1/k}\right)\right) = 1. \]
Now we will show that 
\begin{equation}\label{suff-dim-H}
\dim_H \pi\left( \bigcap_{k \ge 1} \bigcup_{m \ge 1} \bigcap_{n \ge m} A_{n, 1/k} \right) < s. 
\end{equation} 
 
For $0 < t < \epsilon_1$, let 
\[ \mathcal{A}(n, t) := \left\{ (i_1, \dots, i_n) \in (\Sigma_N)^{n} \ \middle| \ \nu_y \left(I(i_1, \dots, i_n) \right) \ge (r_{i_1} \cdots r_{i_n})^s \exp\left(n(\epsilon_1 -t)\right)\right\}. \]
Since 
\[ \sum_{(i_1, \dots, i_n) \in (\Sigma_N)^{n}} \nu_y \left(I(i_1, \dots, i_n) \right) = 1,\]   
we have that 
\begin{equation}\label{Ant-small} 
\sum_{(i_1, \dots, i_n) \in \mathcal{A}(n, t)}  (r_{i_1} \cdots r_{i_n})^s \le \exp\left(-n(\epsilon_1 -t)\right). 
\end{equation}

Fix $k \ge 1$ and $m \ge 1$. 
Denote the diameter of a subset $A$ of $K$  by $\textup{diam}(A)$, 
in other words, 
\[ \textup{diam}(A) = \sup\left\{ d(x,y) | x, y \in A \right\}. \] 
By our assumptions, we have that for some  $c_1 > 0$, 
$$\textup{diam}\left( \pi( I(i_1, \dots, i_n) ) \right) = \textup{diam}(K_{i_1 \cdots i_n}) \le c_1 r_{i_1} \cdots r_{i_n} \le c_1 \left( \max_i r_i \right)^n, \ n \ge 1.$$
Then, we have that for every $n \ge m$, 
\[ \pi\left( \bigcap_{n \ge m}  A_{n, 1/k}\right) \subset \bigcup_{(i_1, \dots, i_n) \in \mathcal{A}(n, 1/k)} \pi \left( I(i_1, \dots, i_n) \right),  \]
and, for every $u \in (0, s)$, 
\begin{align*} 
\sum_{(i_1, \dots, i_n) \in \mathcal{A}(n, 1/k)}  \textup{diam}\left( \pi \left( I(i_1, \dots, i_n) \right) \right)^u
&\le c_1^u \sum_{(i_1, \dots, i_n) \in \mathcal{A}(n, 1/k)}  (r_{i_1} \cdots r_{i_n})^u  \\
&\le c_1^u \sum_{(i_1, \dots, i_n) \in \mathcal{A}(n, 1/k)}  (r_{i_1} \cdots r_{i_n})^s (r_{i_1} \cdots r_{i_n})^{u-s}  \\
&\le  c_1^u  \left(\max_i \frac{1}{r_i}\right)^{(s-u)n} \exp\left(-n\left(\epsilon_1 -\frac{1}{k}\right)\right),  
\end{align*} 
where in the final inequality we have used \eqref{Ant-small} and the assumption that $r_i > 0$.  
Hence, if $k$ is sufficiently large and $u$ is sufficiently close to but strictly smaller than $s$, 
\[ \mathcal{H}_{u} \left( \pi\left( \bigcap_{n \ge m}  A_{n, 1/k}\right)  \right) = 0, \] 
where we let $\mathcal{H}_s$ be the $s$-dimensional Hausdorff measure on $M$.   
Hence, 
\[ \dim_H \pi \left(\bigcup_{m \ge 1} \bigcap_{n \ge m} A_{n, 1/k}\right) \le u < s. \]
Hence \eqref{suff-dim-H} follows. 
\end{proof}

The following is different from a part of the proof of \cite[Theorem 1.2 (ii)]{Ok14}, specifically, \cite[Lemma 3.3]{Ok14}. 

\begin{Lem}\label{key}
Let $y \in Y$. 
Assume that \eqref{multi-sep} holds for $(i_k)_{1 \le k \le l}$.   
Then, 
for every $i \in \N$ and $x \in (\Sigma_N)^{\N}$ satisfying that $X_{i +k}(x) = i_k, 1 \le k \le l$,  
\[ \sum_{k =1}^{l} s_{N}\left(p_{i + k} (y; x)\right) \le  \sup\left\{s_N \left((p_j)_{j \in \Sigma_N} \right) \ \middle| \  \sum_{j} \left|p_j - r_j^s \right| > \epsilon_0  \right\}. \]  
\end{Lem}

If there exist a constant $\displaystyle \epsilon_0 \in (0, \min_i r_i^s)$, and an integer $l \ge 1$ 
such that \eqref{multi-sep} holds for {\it every} $i_1, \dots, i_{l} \in  \Sigma_N$, then 
this inequality holds without the constraint that $X_{i +k}(x) = i_k, 1 \le k \le l$. 

\begin{proof} 
In this proof, $\| \cdot \|$ denotes the $\ell^1$-norm on $\R^l$. We recall that $s_N \left((p_j)_{j \in \Sigma_N} \right) = 0$ if and only if $p_i = r_i^s$ for every $i$. 
  
Case 1.  
Since 
\[ \left\| p_{i} (y; x) - \left(r_0^s, \cdots, r_{N-1}^s\right) \right\| > \epsilon_0, \] 
it holds that 
\[  s_{N}(p_{i} (y; x)) \le \sup\left\{s_N \left((p_j)_{j \in \Sigma_N} \right) \ \middle| \  \sum_{j \in \Sigma_N} \left| p_j - r_j^s \right| > \epsilon_0  \right\}.  \]
Therefore, 
\[ \sum_{k = 1}^{l} s_{N}\left(p_{i + k} (y; x)\right) < \sup\left\{s_N \left((p_j)_{j \in \Sigma_N} \right) \ \middle| \  \sum_{j \in \Sigma_N} \left|p_j - r_j^s \right| > \epsilon_0  \right\}. \]   

Case 2. 
If \[ \left\| p_{i} (y; x) - \left(r_0^s, \cdots, r_{N-1}^s \right) \right\| \le \epsilon_0,   \] 
and moreover $X_{i +k}(x) = i_k, 1 \le k \le l$,
then, by \eqref{multi-sep},      
\[ \left\| p_{i + l} (y; x) - \left(r_0^s, \cdots, r_{N-1}^s\right) \right\|  \ge \left| G_{i + l} (h_{i+l}(y; x)) - r_{i_l}^s \right| > \epsilon_0, \]  
and hence,  
\[ s_{N}\left(p_{i + l} (y; x)\right) \le \sup\left\{s_N \left((p_j)_{j \in \Sigma_N} \right) \ \middle| \  \sum_{j \in \Sigma_N} \left|p_j - r_j^s \right| > \epsilon_0  \right\}.  \]
Therefore, 
\[ \sum_{k = 1}^{l} s_{N}(p_{i + k} (y; x)) < \sup\left\{s_N \left((p_j)_{j \in \Sigma_N} \right) \ \middle| \  \sum_{j \in \Sigma_N} \left|p_j - r_j^s \right| > \epsilon_0  \right\}. \] 
Thus we have the assertion. 
\end{proof}

\begin{Lem}\label{pat}
Let $y \in Y$. 
Then, for every $i_1, \dots, i_l \in \Sigma_N$, 
there exists a non-random constant $c_1 > 0$ such that for $\nu_y$-a.s.$x$, there exists a random subset $I(x) \subset \N$ such that 
\[\liminf_{n \to \infty} \frac{\left| I(x) \cap \{1, \dots, n\} \right|}{n} \ge c_1, \]
and furthermore it holds that $ X_{i +k}(x) = i_k$ for every $i \in I(x)$ and every $1 \le k \le l$.
\end{Lem} 

\begin{proof}
Fix $i_1, \dots, i_l \in \Sigma_N$. 
Let 
\begin{equation}\label{inf-multi}
\widetilde c = \widetilde c(y)  := \inf_{i \in \Sigma_N, z \in Y(y)} G_i (z).
\end{equation}  
By \eqref{A-y}, we have that $\widetilde c (y) > 0$. 
Let $C(n) := \left\{X_{(n-1)l + k} = i_k, 1 \le k \le l \right\}$.  
Let $\widetilde M_0 := 0$ and for $n \ge 1$, 
$\widetilde M_n - \widetilde M_{n-1} := 1_{C(n)} - \widetilde c^{l}$.  
Then, $|\widetilde M_n - \widetilde M_{n-1}| \le 2$, and, 
$\{\widetilde M_n, \F_{ln}\}_n$ is a submartingale. 
Then, by Azuma's inequality \cite{A67},    
\[ \nu_y \left(\sum_{k = 1}^{n} 1_{C(k)} < \frac{n \widetilde c^{l}}{2}\right) 
= \nu_y \left( \widetilde M_n < -  \frac{n  \widetilde c^{l}}{2} \right) \le \exp\left(-\frac{n \widetilde c^{2l}}{32}\right).\]
Hence by the Borel-Cantelli lemma, 
\[ \liminf_{n \to \infty} \frac{1}{n} \sum_{k = 1}^{n} 1_{C(k)} \ge \frac{\widetilde c^{l}}{2}, \textup{ $\nu_y$-a.s.} \]
\end{proof}

\begin{proof}[Proof of Theorem \ref{main} (i)]  
Let $y \in Y$. 
By Lemmas \ref{key} and \ref{pat}, 
there exists $\epsilon_1 > 0$ such that 
\[ \limsup_{n \to \infty} \frac{1}{n} \sum_{k = 1}^{n} s_{N}\left(p_{k} (y; x)\right) \le - \epsilon_1,  \ \textup{ $\nu_y$-a.s.$x \in (\Sigma_N)^{\N}$. } \] 
By this and Proposition \ref{transfer}, 
\[ \limsup_{n \to \infty} \frac{-\log R_{y,n}}{n} \le - \epsilon_1, \ \textup{ $\nu_y$-a.s.} \] 
By this and Lemma \ref{exist}, 
\[ \dim_H \mu_y < s. \]
\end{proof} 

\begin{Rem}\label{A-B-y}
The conclusion also holds if we replace our assumptions with the conditions that 
\begin{equation}\label{A-y} 
0 < \inf_{i \in \Sigma_N, w \in Y(y)}  G_i (w) \le \sup_{i \in \Sigma_N, w \in Y(y)}  G_i (w) < 1, 
\end{equation}
and that \eqref{multi-sep} holds for some $\epsilon_0 \in (0, 1/N)$, $l \ge 1$ and $i_1, \dots, i_{l} \in  \Sigma_N$. 
\end{Rem}

\subsection{Proof of (ii)} 

For $y \in Y$, let 
\begin{equation}\label{widetilde} 
\widetilde R_{y,n}(x) := \nu_y \left( I(X_1 (x), \dots, X_n (x)) \right). 
\end{equation} 

\begin{Lem}\label{any}
Let $y \in Y$. 
Assume that 
\[ \liminf_{n \to \infty} \frac{-\log \widetilde R_{y,n}}{n} \ge a, \textup{ $\nu_y$-a.s.} \]
Then,  
it holds that $\mu_y (K_1) = 0$ for every Borel subset $K_1$ of $K$ such that 
\begin{equation}\label{dim-Haus-small} 
\dim_H (K_1) < \frac{a}{\log (1/r)}.  
\end{equation}   
\end{Lem}

\begin{proof}
Let $n \ge 1$ and $\delta > 0$. 
Assume that \eqref{dim-Haus-small} holds. 
Then, we can take open sets $\{U_{n,l}\}_{l}$ in $K$ such that  for every $l \ge 1$, 
$$\textup{diam}(U_{n,l}) \le c r^{n},$$   
\begin{equation}\label{K-covered} 
K_1 \subset \bigcup_{l \ge 1} U_{n,l}, 
\end{equation}  
and 
\begin{equation}\label{sum-diam-delta} 
\sum_{l \ge 1} \textup{diam}(U_{n,l})^{a/\log (1/r)} < \delta.  
\end{equation} 

Let $k(n,l)$ be an integer such that 
$$c r^{k(n,l)} \le \textup{diam}(U_{n,l}) < c r^{k(n,l)-1}.$$
Then, by $k(n,l) \ge n$, 
we have that 
\[ \nu_y \left(I_{k(n,l)}(x)\right)  \le \exp(-k(n,l) a) \le \textup{diam}(U_{n,l})^{a/\log (1/r)}  \]
for every $x \in \bigcap_{k \ge n} \{-\log \widetilde R_{y,k} \ge k a\}$.  
By this and the assumption of (ii), we have that 
\[ \nu_y \left( \pi^{-1}  (U_{n,l}) \cap \bigcap_{k \ge n} \{-\log \widetilde R_{y,k} \ge k a\}\right) \le D \textup{ diam}\left(U_{n,l}\right)^{a/\log (1/r)}.  \]

By this and \eqref{K-covered} and \eqref{sum-diam-delta}, we have that 
\[  \nu_y \left( \pi^{-1}  (K_1) \cap \bigcap_{k \ge n} \{-\log \widetilde R_{y,k} \ge k a\}\right) \le D \delta. \]

By this, \eqref{mu-nu-pi}  and the assumption, it holds that 
$\mu_y (K_1) = 0$.     
\end{proof} 

By Lemma \ref{any}, Proposition \ref{transfer}, and the assumption of (ii),   
\[ \dim_H \mu_y \ge  \frac{-c\log c - (1-c) \log(1-c)}{\log (1/r)} > 0,  \]
where $c$ is the constant in \eqref{wAy}.  
This completes the proof of Theorem \ref{main} (ii). 

\section{Examples}

This section is devoted to state various examples. 
All examples considered here satisfy that $(M,d)$ is the Euclid space $\mathbb{R}^d$ and  $\{S_i\}_i$ are contractive similitude on $\mathbb{R}^d$ satisfying the open set condition, and $Z = K$ and $F_i = S_i, i \in \Sigma_N$, where $K$ is the attractor of $\{S_i\}_i$.  
We recall that $\dim_H K = s$ for $s$ satisfying that 
\[ \sum_{i \in \Sigma_N} \textup{Lip}(S_i)^s = 1.\]
{\it Unless otherwise stated, we deal with the homogeneous case, i.e. all Lipschitz constants $r_i, i \in \Sigma_N$ are equal to each other and furthermore $\{S_i\}_i$ satisfies the open set condition.}  
In the final subsection, we deal with the {\it non-}homogeneous case.   

As we can see in \cite[Corollary 1.3]{Kig95} and its subsequent remarks, 
Theorem \ref{main} is applicable to the case that $(K, \{S_i\}_i)$ defines a self-similar set satisfying the open set condition such as the Sierp\'inski  gasket, the Sierp\'inski  carpet, the Koch curve. 
Theorem \ref{main} (ii) is applicable to the case that it is difficult to check whether $\{S_i\}_i$ satisfies the open set condition such as the L\'evy curve. 
See \cite[Appendix]{Kig95} for more details. 
Let $V_0$ be the set of fixed points of $S_i, i \in \Sigma_N$. 
Let $\displaystyle V_m := \cup_{i_1, \dots, i_m \in \Sigma_N} S_{i_1, \dots, i_m}(V_0)$.
Let $\displaystyle r_0 := \max_{i \in \Sigma_N} \textup{Lip}(S_i) \in (0,1)$.

\begin{Lem}
Assume that there exists $D > 1$ such that for every large $n$, 
\begin{equation}\label{fin-num} 
\sup_{x \in V_n} \left| V_n \cap B(x, 2 r_0^n \textup{diam}(K)) \right| \le D.  
\end{equation} 
Then, there exist two constants $c > 0$ and $D^{\prime} > 0$ such that for every $x \in K$ and $m \ge 1$, 
$$\left| \left\{ (i_1 \cdots i_m) \in (\Sigma_N)^{m}  \ \middle| \   B(x, c r_0^m) \cap  K_{i_1 \cdots i_m} \ne \emptyset \right\} \right| \le D^{\prime}.$$
\end{Lem}

\begin{proof}
Assume $\textup{diam}(U) \le  \textup{diam}(K) r_0^{n}$. 
Let $x, y \in U$. 
Then there exist points $x_n, y_n \in V_n$ such that $\max\{d(x, x_n), d(y, y_n)\} \le r_0^n \textup{diam}(K)$. 
Then $d(x_n, y_n) \le 2 r_0^n \textup{diam}(K)$. 
By the assumption, there are at most $D$ sets of forms $S_{i_1, \dots, i_n}(K)$ covering $U$. 
\end{proof}

\begin{Exa}\label{exa-cover}
We consider the Euclid metric. \\
(i) If $S_i (z) = (z+i)/N$, $z \in \R$, then, $K = [0,1]$, and, \eqref{fin-num} holds for $r_0 = 1/N$. \\
(ii) If $N = 4$, $S_i (z) = (z+q_i)/2$, $z \in \R^2$, where $$\{q_i\}_i = \{(x,y) \in \Z^2 \mid 0 \le x, y\le 1\},$$
then, $K = [0,1]^2$, and, \eqref{fin-num}  holds for $r = 1/2$.\\
(iii) If $N = 3$, $S_i (z) = (z+q_i)/2$, $z \in \R^2$, where $\{q_i\}_i$ forms an equilateral triangle, 
then, $K$ is a 2-dimensional  Sierp\'inski  gasket, and, \eqref{fin-num} holds for $r = 1/2$.\\
(iv) If $N = 8$, $S_i (z) = (z+q_i)/3$, $z \in \R^2$, where $$\{q_i\}_i = \{(x,y) \in \Z^2 \mid 0 \le x, y\le 2\} \setminus \{(1,1)\},$$
then,  $K$ is a 2-dimensional  Sierp\'inski  carpet, and, \eqref{fin-num} holds for $r = 1/3$.
\end{Exa}

\begin{proof}
Let $\displaystyle L := \left\{\sum_{i} a_i q_i : a_i \in \Z \right\}$.
This is a discrete subset of $[0,1]$ or $\R^2$.  
By the definition of $S_i$, 
it holds that for every $n$, 
$V_n \subset N^{-n} L$. 
Hence, $\displaystyle \sup_{x, y \in V_n, \ x \ne y} d(x,y) \ge c N^{-n}$, and, \eqref{fin-num} holds for some $D$.   
\end{proof} 

If $Y$ is a one-point set, then, \eqref{g-me} does not depend on $y$, so we drop the notation, 
and then, the solution $\mu$ of \eqref{g-me} is an invariant measure of the iterated function system $\left(K, \{S_i\}_{i \in \Sigma_N}\right)$ with weights $\{G_i\}_{i \in \Sigma_N}$.          

Hereafter, if $(G_0, \dots, G_{N-1}) = (p_0, \dots, p_{N-1})$, 
then, we call $\nu$ and $\mu$ the $(p_0, \dots, p_{N-1})$-{\it Bernoulli measure} on $Z = (\Sigma_N)^{\mathbb{N}}$ and the $(p_0, \dots, p_{N-1})$-{\it self-similar measure} on $Z = K$, respectively.    
We denote them by $\nu_{(p_0, \dots, p_{N-1})}$ and $\mu_{(p_0, \dots, p_{N-1})}$, respectively. 

\subsection{Singularity with respect to self-similar measures}

\begin{Prop}[Singularity with respect to self-similar measures]\label{singu-ss} 
Let $p_{0}, \dots, p_{N-1}$ be positive numbers satisfying that $\sum_{i \in \Sigma_N} p_i = 1$. 
Assume that for some $i_1, \dots, i_l$, 
\begin{align}\label{multi-sep-2}
Y(y) &\cap \bigcap_{i \in \Sigma_N} G_{i}^{-1}\left( \left[p_i - \epsilon_0,  p_i+ \epsilon_0\right]\right) \notag\\
&\cap \bigcap_{j \in \Sigma_N} \left(G_{j} \circ H_{i_{l}} \circ \cdots \circ H_{i_{1}}\right)^{-1} \left( \left[ p_j - \epsilon_0,  p_j + \epsilon_0\right]\right) = \emptyset. 
\end{align}
Then, (i) $\nu_y$ is singular with respect to $\nu_{(p_{0}, \dots, p_{N-1})}$.\\
(ii) If moreover $\pi^{-1}(\pi(A)) \setminus A$ is at most countable for every subset $A$ of $(\Sigma_N)^{\N}$, 
$\mu_y$ is singular with respect to $\mu_{(p_{0}, \dots, p_{N-1})}$.   
\end{Prop} 

\begin{proof} 
By Azuma's inequality, 
$\nu_{p_0, \dots, p_{N-1}}$-a.s.$x$, 
there are infinitely many $i$ such that 
$X_{i + k} (x) = i_k$ for every $1 \le k \le l$.   
By \eqref{multi-sep-2}, 
$\nu_{p_0, \dots, p_{N-1}}$-a.s.$x$, 
there are infinitely many $i$ such that 
\[ \left\| p_i (y; x) - (p_0, \dots, p_{N-1}) \right\| \ge \epsilon_0. \]

Now assertion (i) follows from this and \cite[Theorem 4.1]{Hi04}.   
Assertion (ii) follows from assertion (i), \eqref{mu-nu-pi} and the assumption. 
\end{proof}

\subsection{Energy measures on Sierpi\'nski gaskets} 

\cite{Ku89} shows that energy measures for canonical Dirichlet forms on Sierpi\'nski gaskets are singular with respect to the Hausdorff measure on them.   
It is generalized by \cite{BST99}, \cite{Hi04}. 
Recently, \cite{JOP17} considers the Kusuoka measure from an ergodic theoretic viewpoint.  
Their framework covers a general class of measures that can be defined by products of matrices.

Let $V_0 := \{q_0, q_1, q_2\}$ be the set of vertices of an equilateral triangle in $\R^2$.  
Let $K$ be a 2-dimensional Sierpi\'nski gasket,   
that is, the attractor of $K = \cup_{i = 0,1,2} S_i (K)$,
where we let $S_i (z) := (z + q_i)/2, z \in \R^2$.    
$V_0$ is the set of fixed points of $S_i, i = 0,1,2$. 

Let $a_i \in K$, $i = 0,1,2$, be unique fixed points of $F_i$, $i = 0,1,2$, respectively.    
Let 
\[ A_0 := \begin{pmatrix} 3/5 & 0 \\ 0 & 1/5 \end{pmatrix},  
A_1 := \begin{pmatrix} 3/10 & \sqrt{3}/10 \\ \sqrt{3}/10  & 1/2 \end{pmatrix},  
A_2 := \begin{pmatrix} 3/10 & -\sqrt{3}/10 \\ -\sqrt{3}/10  & 1/2 \end{pmatrix}. \]
They are regular matrices and define linear transformation of $Y$. 

Let $\| \cdot \|$ be the Euclid norm on $\R^2$. 
Let 
\[ Y := S^1 = \{x \in \R^2 : \| x \| = 1\}.\]  
We regard $Y$ as a topological space with respect to the Euclid distance on $Y \subset \R^2$.
For $y \in Y$ and $i = 0,1,2$, let 
\[ G_i (y) := \frac{5}{3} \|A_i y\|^2, \textup{ and, }  H_i (y) := \frac{A_i y}{\| A_i y \|}. \]

\begin{Lem}
(i) \[ A_0^2 + A_1^2 + A_2^2 = \frac{3}{5} I_2,\] 
where $I_2$ denotes the identity matrix.  
In particular, \eqref{G-sum-1} holds.\\ 
(ii) For every $y$,  
\begin{equation*} 
0 < \inf_{i \in \Sigma_N, z \in Y(y)}  G_i (z) \le \sup_{i \in \Sigma_N, z \in Y(y)}  G_i (z) < 1. 
\end{equation*}
\end{Lem} 

\begin{proof}
(i) is immediately seen.  

(ii) The set of eigenvalues of $A_0, A_1, A_2$ are $\{1/5, 3/5\}$. 
Hence, for every $i$ and $y$, 
\[ \frac{1}{15} \le G_i (y) \le \frac{3}{5}.\] 
\end{proof}

\begin{Lem}
(i) If $G_0 (y) = G_0 \circ H_0 (y)$, then, $y \in \{(\pm 1, 0), (0, \pm 1)\}$ and furthermore 
\[ G_0 (y) = G_0 \circ H_0 (y) \in \left\{\frac{1}{15}, \frac{3}{5}\right\}. \]   
In particular, if $G_0 (y) = 1/3$, then, \[ G_0 \circ H_0 (y) \ne \frac{1}{3}. \]
(ii) For every $y \in Y$,  
there exist $\epsilon_0 \in (0, 1/N)$, $l \ge 1$ and $i_1, \dots, i_{l} \in  \Sigma_N$ such that \eqref{multi-sep} holds. 
\end{Lem}  

\begin{proof}
Assertion (i) is easy to see. 
By using assertion (i) and that fact that $G_i$ and $H_i$ are continuous on $Y$ and $Y$ is compact, assertion (ii) follows. 
\end{proof} 

\begin{Lem}
Assume $G_i (y) = p_i \in (0,1)$, $i = 0,1,2$. 
Then,\\
(i) If $y \notin \{(\pm 1, 0), (0, \pm 1)\}$,  
then, 
\[ G_0 (H_0 (y)) \ne p_0. \] 
(ii)  If $y \in \{(\pm 1, 0), (0, \pm 1)\}$,  
then, 
\[ G_0 (H_1 (y)) \ne p_0.\]   
\end{Lem} 

Now we can apply Theorem \ref{main} and Proposition \ref{singu-ss} to this case, ($N = 3, l = 2$) 

\begin{Prop}
It holds that 
\begin{equation}\label{KEM-singular} 
0 < \dim_H \mu_{y} < \frac{\log 3}{\log 2}. 
\end{equation}  
Furthermore, $\mu_y$ is singular with respect to every $(p_0, p_{1}, p_{2})$-self-similar measure on $K$. 
\end{Prop}

\begin{Rem}
This result could be deduced from the dimension theory of dynamical systems. 
By \cite[Remark 1.6, Corollary 2.12, and Example 1]{Ku89}, 
we have that on the symbolic space $(\Sigma_N)^{\mathbb{N}}$ the Kusuoka measure is invariant and ergodic with respect to the canonical shift.  
See \cite[Theorem 6.8]{Kaj12} for a simple proof. 
By the uniqueness of measures of maximal dimension stated in \cite[Theorem 4.4.7]{MU03} and mutual singularity of two distinct ergodic measures (This is an  easy consequence of the Birkhoff ergodic theorem. See \cite[Theorem 2.2.6]{PU10} for example), 
we have the upper bound of \eqref{KEM-singular}.  
\cite{JOP17} investigates the Kusuoka measure from an ergodic theoretic viewpoint, which could be seen as a generalization of Bernoulli measures.
\end{Rem}

Let $f$ be a harmonic function on $K$. 
Let $h_1$ and $h_2$ be the harmonic functions on $K$ such that 
\[ (h_1(q_0), h_1(q_1), h_1(q_2)) =  (0, \sqrt{2}, \sqrt{2}) \]
 and \[ (h_2(q_0), h_2(q_1), h_2(q_2)) =  (0, \sqrt{2/3}, -\sqrt{2/3}).\]   
Let $v$ be  the components of $f$ in $(h_1, h_2)$.   
Let $y = v/\|v\|$. 
Then, the energy measure associated with $f$ is $\mu_y$. (Cf. \cite{BST99}.)  

\begin{Rem}
In a formal level, the framework adopted in \cite{Hi04} is interpreted as follows.    
See \cite[Section 2]{Hi04} for details of Dirichlet forms.  
Let $(K, \Sigma_N, \{\psi_i\}_{i \in \Sigma_N})$ be a self-similar structure and $\mu$ be the invariant measure.  
Let $(\mathcal{E}, \mathcal{F})$ be a regular Dirichlet form on $L^2 (K, \mu)$. 
Assume (A1)-(A6) in  \cite[Section 2]{Hi04}.   
\[ Y := \mathcal{F}, \]
\[ G_i (f) :=  s_i \frac{\mathcal{E}(f \circ \psi_i, f \circ \psi_i)}{\mathcal{E}(f, f)}, \ \textup{ if } \mathcal{E}(f, f) > 0.  \]
\[ G_i (f) :=  \frac{1}{N}, \ \textup{ if } \mathcal{E}(f, f) = 0.  \]
\[ H_i (f) := f \circ \psi_i.  \]
Then, for $f \in Y$, $\mu_f$ is the normalized energy measure.  
\end{Rem} 

\subsection{Restriction of harmonic function on Sierpi\'nski gasket}  

\cite{ADF14} considers  the restriction on $[0,1]$ of harmonic functions on the Sierpi\'nski gasket, 
and shows that they are singular functions\footnote{A singular function is a continuous, increasing function on $[0,1]$ whose derivative is zero almost surely with respect to the Lebesgue measure.} whenever they are monotone. 
The restrictions are among a wider class of functions containing several functions such as Lebesgue singular functions.  

\cite[Notation 6]{ADF14} is interpreted as follows in our framework: 
Let 
\begin{equation}\label{def-ADF-Y} 
Y = [0,1], 
\end{equation}
\begin{equation}\label{def-ADF-G}  
G_0 (y) = \frac{2+y}{5}, \ \ G_1 (y) = 1 - G_0 (y) = \frac{3-y}{5}, 
\end{equation}  
and 
\begin{equation}\label{def-ADF-H}  
H_0 (y) = \frac{1+2y}{2+y}, \ \ H_1 (y) = \frac{y}{3-y}. 
\end{equation} 
Then, $S_y$ in \cite[Subsection 2.1]{ADF14} is the distribution function of $\mu_y$.  
\cite[Theorem 3]{ADF14} is their main theorem, and,  
\cite[Lemma 24 and Theorem 25]{ADF14} restates it in a more general framework.       

This is {\it not} of de Rham type in the subsection below, even if we exchange $H_0$ with $H_1$.  
Theorem \ref{main} (i) gives the following improvements for \cite[Theorem 25]{ADF14}. 
\eqref{A-y} in Remark \ref{A-B-y}  corresponds to \cite[assumption (a) in Lemma 24]{ADF14}. 
The assumption that  \eqref{multi-sep} holds for $\epsilon_0 \in \left(0, (\min_i r_i)^s\right)$, $l \ge 1$ in Theorem \ref{main} (i) corresponds to \cite[assumption (b) in Lemma 24]{ADF14}. 

We loosen the assumptions in two ways and simultaneously obtain a stronger conclusion. 
The first way for weakening the assumption is adopting the assumption that \eqref{multi-sep} holds for some $\epsilon_0 \in (0, (\min_i r_i)^s)$, $l \ge 1$ and $i_1, \dots, i_{l} \in  \Sigma_N$, which is strictly weaker than \cite[assumption (b) in Lemma 24]{ADF14}. 
The second way is the case that \eqref{A-y} fails but \eqref{multi-sep} holds for $\epsilon_0 \in \left(0, (\min_i r_i)^s\right)$, $l \ge 1$. 

\begin{Thm}[Application to singularity for real functions]
Let $Y = Z = [0,1]$ and $S_i (z) = (z+i)/N, z \in [0,1], i \in \Sigma_N$.   
Let $y \in [0,1]$. 
Assume $\mu_y$ has no atoms.  
Let $\varphi_y$ be the distribution function of $\mu_y$. 
Then,\\
(i) If \eqref{A-y} holds and  \eqref{multi-sep} holds for some $\epsilon_0 \in (0, 1/N)$, $l \ge 1$ and $i_1, \dots, i_{l} \in  \Sigma_N$, then, $\dim_H \mu_y < 1$.\\
(ii) If \eqref{multi-sep} holds for some $\epsilon_0 \in (0, 1/N)$, $l \ge 1$ and $i_1, \dots, i_{l} \in  \Sigma_N$, 
then, $\varphi_y$ does not have non-zero derivative at almost every point with respect to every $(p_0, \dots, p_{N-1})$-self-similar measure.\\
(iii) If \eqref{multi-sep} holds for $\epsilon_0 \in \left(0, 1/N\right)$, $l \ge 1$, 
then, $\varphi_y$ does not have non-zero derivative at every point, and, $\dim_H \mu_y < 1$. 
\end{Thm} 

\begin{proof}
Assume that $\varphi_y$ has non-zero derivative at $\pi(x) \in (0,1)$.    
Then, 
\begin{equation}\label{nonzero-deri} 
\lim_{n \to \infty} G_{X_{n}(x)} \circ H_{X_{n-1}(x)} \circ \cdots \circ H_{X_1(x)}(y) = \frac{1}{2}. 
\end{equation} 

Assume that \eqref{multi-sep} holds for some $\epsilon_0 \in (0, 1/N)$, $l \ge 1$ and $i_1, \dots, i_{l} \in  \Sigma_N$. 
Let $\mu_{p_0, \dots, p_{N-1}}$ be a Bernoulli measure on $\{0,1\}^{\N}$.      
Then by Azuma's inequality, 
\[ \mu_{p_0,  p_{1}}\left( \bigcap_{n \ge 1} \bigcup_{m \ge n} \bigcap_{k = 1}^{l} \{X_{m+k} = i_{k}\} \right) = 1.\]   

By this, \eqref{nonzero-deri} fails for $\mu_{p_0, \dots, p_{N-1}}$-a.e.$x \in \{0,1\}^{\N}$.   
Thus (ii) follows. 
The assumption in assertion (iii) implies that \eqref{nonzero-deri} fails for {\it every} $x \in \{0,1\}^{\N}$. 
\end{proof} 

\begin{Rem}
If we see the proof, 
assertion (ii) above holds even if we replace an arbitrarily Bernoulli measure with an arbitrarily measure satisfying that there exists $c \in (0,1)$ such that for every $m, k \ge 1$, 
\[ c\mu\left(\left[\frac{k-1}{2^m}, \frac{k}{2^m}\right)\right) \le \mu\left(\left[\frac{k-1}{2^m}, \frac{2k-1}{2^{m+1}}\right)\right) \le (1-c)\mu\left(\left[\frac{k-1}{2^m}, \frac{k}{2^m}\right)\right). \]
\end{Rem} 

In the same manner as in the proof of assertion (ii) of Proposition \ref{singu-ber-dR} below, 
\begin{Prop}[Singularity with respect to self-similar measures]  
If \eqref{def-ADF-Y}, \eqref{def-ADF-G} and \eqref{def-ADF-H} hold, 
then, $\mu_y$ is singular with respect to every $(p_0, p_1)$-self-similar measure.  
\end{Prop}

\begin{Rem} 
Our main results might be applicable to different models of Sierp\'inski  gaskets such as level 3 Sierp\'inski  gaskets.    
\cite{ADF14} deals with the two-dimensional standard Sierp\'inski  gaskets.   
\end{Rem} 

\begin{Rem}
The restrictions of harmonic functions on the attractors of different iterated function systems have also been considered.    
\cite[Theorem 3.2]{ES16} states that the restrictions of harmonic functions on the Hata tree is singular. 
Assume the framework of \cite[Theorem 3.2]{ES16}.  
Denote a restriction by $\varphi$, and the measure whose distribution function is $\varphi$ by $\mu_{\varphi}$. 
Then, 
\begin{equation}\label{dR-tree}
	\varphi(z) = \begin{cases}
						\dfrac{1}{|h|^2} \varphi\left( \dfrac{z}{|\alpha|^{2}} \right) & 0 \le z \le |\alpha|^2, \\
					        \left(1- \dfrac{1}{|h|^2} \right) \varphi\left( \dfrac{z - |\alpha|^{2}}{1 - |\alpha|^{2}} \right) + \dfrac{1}{|h|^2} & |\alpha|^2   \le z \le 1,     
				          \end{cases}
\end{equation}
This appears in the proof of \cite[Theorem 3.2]{ES16}\footnote{There is a typo in it, and it will be fixed in \eqref{dR-tree}.}.
See also Section 4 for functional equations of this kind.  

Then, $\mu_{\varphi}$ satisfies \eqref{g-me} for $Y = \{\textup{one-point set}\}$, $G_0 = |h|^{-2}$, $Z = [0,1]$, 
\[ F_0 (z) = |\alpha|^2 z, \textup{ and } F_1 (z) = (1-  |\alpha|^2 )z +  |\alpha|^2.\]      
(The notation $F_1$ here is different from \cite{ES16}.) 
It is known that 
\[ \dim_H \mu_{\varphi} = \frac{s_2 \left(|h|^{-2}, 1- |h|^{-2}\right) + \log 2}{-  |h|^{-2} \log |\alpha|^2 - (1-  |h|^{-2}) \log (1- |\alpha|^2)}. \]
Hence, if $a \ne b$, then, $\dim_H \mu_{\varphi} < 1$ and hence $\varphi$ is a singular function.   
Further, multifractal analysis for $\mu_{\varphi}$ is investigated. 
See \cite{F97} for details. 
\end{Rem} 

\subsection{De Rham's functional equations} 

Before we apply our results to a class of de Rham's functional equations, we give a short review.  
De Rham \cite{dR56, dR57}\footnote{An English translation of \cite{dR57} is included in Edgar \cite{E93}.} considered a certain class of functional equations. 
He considered the solution $\varphi$ of the following functional equation which takes its values in a certain metric space: 
\begin{equation}\label{dR-N}
	\varphi(x) = \begin{cases}
						g_{0}(\varphi(Nx)) & 0 \le x \le 1/N, \\
						g_{1}\left(\varphi(Nx-1)\right) & 1/N  \le x \le 2/N, \\
						\cdots\\
						g_{N-1}\left(\varphi(Nx-(N-1))\right) & (N-1)/N  \le x \le 1,     
				          \end{cases}
\end{equation}
where each $g_i$ is a weak contraction (its definition is given below), $g_0 (0) = 0$, $g_{N-1}(1) = 1$ and $g_{i+1} (0) = g_i (1)$ for each $i$.  
Solutions of de Rham's functional equations give parameterizations of several self-similar sets such as the K\^och curve and the P\'olya curve, etc.   
Some singular functions such as the Cantor, Lebesgue, etc. functions are solutions of such functional equations. 

We give a short review of some known results.   
\cite{BK00} considers self-similarity, inversion and composition of de Rham's functions, and points out a connection with Collatz's problem.  
\cite{Kr09} shows connections between sums related to the binary sum-of-digits function and the Lebesgue's singular function, and its partial derivatives with respect to the parameter.  
\cite{Kaw11} investigates the set of points where Lebesgue's singular function has the derivative zero. 
\cite{P04} regards a de Rham curve as the limit of a polygonal arc by repeatedly cutting off the corners, 
obtain a formula for the local H\"older exponent of a de Rham curve at each point, and describe the sets of points with given local regularity.   
\cite{Be08, Be12} consider multifractal and thermodynamic formalisms for the de Rham function. 
Recently, 
\cite{BKK18} performs the multifractal analysis for the pointwise H\"older exponents of zipper fractal curves on $\mathbb{R}^d$ generated by affine mappings on $\mathbb{R}^d$, $d \ge 2$. 
\cite{BKK18} and \cite{N04} consider the Hausdorff dimension of the image measure of the Lebesgue measure on an interval by the de Rham function.  
\cite[Section 9.3]{PV17} also consider de Rham curves in terms of matrix products. 
\cite{DL91, DL92-1, DL92-2, P06} are related to wavelet theory.     
The length of de Rham curve is investigated in \cite{Me98} and \cite{DMS98}. 
\cite{TGD98} uses de Rham's functional equations to construct the Sinai-Ruelle-Bowen measures for several Baker-type maps. 
de Rham's functional equations also appear in \cite{DS76}, which is concerned with gambling strategy. 
\cite{G93, Z01, SB17, Ok19+} consider generalizations of de Rham's functions. 

The case that some of $g_i$ are {\it not} affine is considered in \cite{Ha85, SLK04, Ok14, Ok16, Ok19+, SB17}, 
\cite{SLK04} considers it from a view point of random walk.  
\cite{Ha85} and \cite{SLK04} use the technique of \cite{La73}.  
 \cite{SB17} gives conditions for existence, uniqueness and continuity, and furthermore, provides a general explicit formula for contractive systems. 
However, it seems that real-analytic properties for the solutions are not fully investigated.     

Let $(M, d)$ be a metric space. 
Following \cite[Definition 2.1]{Ha85}, 
we say that a function $f : M \to M$ is a {\it weak contraction} if 
for every $t > 0$, 
\[ \lim_{s \to t, s > t} \sup_{d(x,y) \le \delta} d(f(x), f(y)) < t. \] 

By \cite[Corollary 6.6]{Ha85}, 
if each $g_i$ is a weak contraction, $g_0 (0) = 0$, $g_{N-1}(1) = 1$ and $g_{i+1} (0) = g_i (1)$ for each $i$, 
then, \eqref{dR-N} has a unique continuous solution $\varphi$ and 
we let $\mu = \mu_{\varphi}$ be the probability measure such that the solution $\varphi$ of \eqref{dR-N} is the distribution function of $\mu_{\varphi}$.
If $\dim_H \mu_{\varphi} < 1$, then, $\mu_{\varphi}$ is singular with respect to the Lebesgue measure, and hence $\varphi$ is a singular function on $[0,1]$. 
Therefore, it is valuable to know whether  $\dim_H \mu_{\varphi} < 1$ or not.

\subsubsection{Dimension formula}

Contrary to the examples in the above subsections, in this framework, $\mu_y$ is the invariant measure of a certain iterated function system with place-dependent probabilities,   
and furthermore, we have a form of dimension formula for $\mu_{\varphi}$ for a large class of $(g_i)_i$, thanks to Fan-Lau \cite{FL99}. 
By the Lebesgue-Stieltjes integral, for every bounded Borel measurable function $F : [0,1] \to \mathbb{R}$, 
\[ \int_{[0,1]} F(x) d\mu_{\varphi}(x) = \sum_i \int_{[0,1]} g_i^{\prime}(\varphi(x)) F\left(\frac{x+i}{N}\right) d\mu_{\varphi}(x).  \]  

We assume that $g_i \in C^2([0,1])$ and $0 < g_i^{\prime}(z) < 1$ hold for each $i \in \Sigma_N$ and every $z \in [0,1]$. 
Then, $\varphi$ is H\"older continuous and 
\[ \int_0^1 \frac{\sup\left\{\left| \log g_i^{\prime} (\varphi(s_1))  - \log g_i^{\prime} (\varphi(s_2)) \right| : |s_1 - s_2| \le t\right\}}{t} dt \le \sup_{s \in [0,1]} \frac{|g_i^{\prime\prime}(s)|}{g_i^{\prime}(s)} \int_0^1 t^{c-1} dt < +\infty, \]
where we let $c = c_{\varphi}$ be a positive number.  

Let $h$ be a positive function on $[0,1]$ such that 
\[ h(x) = \sum_{i \in \Sigma_N} g_i^{\prime}\left(g_i (\varphi(x))\right) h \left(\frac{x+i}{N}\right).  \]
\[ h \circ \varphi^{-1} (x) = \sum_i g_i^{\prime}(g_i (x)) h \circ \varphi^{-1} (g_i (x)).  \]
This is unique under the constraint that 
\[  \int_{[0, 1]}  h(x) \mu_{\varphi}(dx) = 1. \] 
See \cite[Theorem 1.1]{FL99}.  

Then, by \cite[Corollary 3.5]{FL99},  
\[ \dim_H \mu_{\varphi} 
= \frac{\sum_{i \in \Sigma_N} \int_{[i/N, (i+1)/N]} h(x) \log \left(1 / g_i^{\prime}(\varphi (x)) \right) \mu_{\varphi}(dx)}{\log N}. \]
We have that 
\[ \dim_H \mu_{\varphi} 
= \frac{\sum_{i \in \Sigma_N}\int_{[0, 1]} H(g_i (y)) g^{\prime}_i (y) \log \left(1/g_i^{\prime}(g_i(y))\right) \ell(dy)}{\log N}, \]  
where $\ell$ is the Lebesgue measure on $[0,1]$ and $H$ is a function on $[0,1]$ satisfying the following conditions: 
\begin{equation}\label{fe-altered} 
H(y) = \sum_{i \in \Sigma_N} g_i^{\prime}(g_i (y)) H(g_i (y)),  \textup{ and } \int_{[0, 1]} H(y) \ell(dy) = 1. 
\end{equation}  

It is interesting to investigate properties for $H$. 
If each $g_i$ is linear, in other words, $g_i^{\prime}$ is a constant function, 
then, $\sum_{i \in \Sigma_N} g_i^{\prime} = 1$, and hence, 
$H \equiv 1$ satisfies \eqref{fe-altered}. 

We focus on the case that $g_i$ is {\it not} affine. 
In that case, it is difficult  for knowing whether 
\[ \sum_{i \in \Sigma_N}\int_{[0, 1]} H(g_i (y)) g^{\prime}_i (y) \log \left(1/g_i^{\prime}(g_i(y))\right) \ell(dy) < \log N. \]  
 
In the following subsection, we give a necessary and sufficient condition for $\dim_H \mu_{\varphi} < 1$ for a specific choice for $(g_i)_i$ by using Theorem \ref{main}.      
It is interesting to investigate properties for $H$, but in this paper we do not analyze $H$ directly.  

Furthermore, by \cite[Theorem 1.6]{FL99}, 
\begin{equation}\label{limit-exist}
\lim_{n \to \infty} \frac{-\log \mu_{\varphi}(I_n (x))}{n} = \sum_{i \in \Sigma_N} \int_{[0,1]} h(z)\log \left(1/g_i^{\prime}(\varphi(z))\right) d\mu_{\varphi}(z),  \textup{ $\mu_y$-a.s.$x \in [0,1]$,} 
\end{equation} 
where we let $I_n (x) := [i/2^n, (i+1)/2^n)$ such that $x \in [i/2^n, (i+1)/2^n)$.  
However, we do not see how to estimate the integrand in the right hand side of \eqref{limit-exist}. 
Arguments in \cite{FL99} depend on the fact $\mu_y$ is an invariant measure of an iterated function system, and we are not sure whether a convergence corresponding to \eqref{limit-exist} holds for the examples in the above subsections.  

\subsubsection{De Rham's functional equations driven by $N$ linear fractional transformations}

\cite{Ok14} considers de Rham's functional equations driven by {\it two} linear fractional transformations, 
here we consider not only the case that $N = 2$ and but also the case that $N \ge 3$.  
Our outline is similar to the one in \cite{Ok14}, however several additional considerations are needed.  

In the following, we consider the equation \eqref{dR-N} for the case that all $g_{i}$ are {\it linear fractional transformations}. 
Let $\Phi(A ; z) := \dfrac{az + b}{cz + d}$ for a $2 \times 2$ real matrix $A = \begin{pmatrix} a & b  \\ c & d \end{pmatrix}$ and $z \in \mathbb{R}$.
Let \[ g_{i}(x) := \Phi\left(A_{i}; x\right), \ x \in [0,1], \ i = 0, 1, \dots, N-1,\] 
such that $2 \times 2$ real matrices $A_{i} = \begin{pmatrix} a_{i} & b_{i}  \\ c_{i} & d_{i} \end{pmatrix}$, $i = 0,1$, satisfy the following conditions (A1) - (A3).\\
(A1) $\Phi(A_{0}; 0) = 0$, $\Phi(A_{N-1}; 1) = 1$, and \\
\[ \Phi(A_{i-1} ; 1) =\Phi(A_{i} ; 0), \ i = 1, \dots, N-1.\]   
(A2) \[ \det A_i = a_{i}d_{i} - b_{i}c_{i} > 0, \ i = 0,1, \dots, N-1.\]
(A3) \[ (a_{i}d_{i} - b_{i}c_{i})^{1/2} \le \min\{d_{i}, c_{i}+d_{i}\}, \ i = 0,1, \dots, N-1.\] 

In several cases, we will replace (A3) with a stronger condition (sA3):  
\[ (a_{i}d_{i} - b_{i}c_{i})^{1/2} < \min\{d_{i}, c_{i}+d_{i}\}, \ i = 0,1, \dots, N-1. \] 

By (A2) and (A3),  $d_i > 0$, and henceforth, 
we can assume $d_i = 1$ for each $i$, without loss of generality.  
By (A1), $b_0 = 0$, and, $0 < b_i < 1$, $1 \le i \le N-1$. 
By (A3), $a_0 \le 1$. 

\begin{Lem}
For each $i$, $\Phi\left(A_{i}; x\right)$ is a weak contraction on $[0,1]$.  
\end{Lem} 

\begin{proof}
If $c_i = 0$, then, 
$\Phi\left(A_{i}; x\right) = a_i x + b_i$. 
By (A2) $a_1 > 0$. 
By (A1), $a_i + b_i = b_{i+1} < 1$ and $b_i \ge 0$ for $i \le N-2$, and,  $a_i + b_i \le 1$ and $b_i > 0$ for $i \ge 1$.   
Hence $a_i < 1$. 
Thus $\Phi\left(A_{i}; x\right)$ is a contraction on $[0,1]$.  

If $c_i \ne 0$,  
then, by (A3), for every $t > 0$,  
\[ \min_{t < |x-y|} \frac{a_i - b_i c_i}{(c_i x + 1)(c_i y + 1)} < 1.  \]
Hence, $\Phi\left(A_{i}; x\right)$ is a weak contraction.  
\end{proof} 

Therefore, 
(A1) - (A3) guarantee that \eqref{dR-N} has a unique continuous solution $\varphi$. 
The above (sA3) is identical with (A3) in \cite{Ok14}.  

We remark that our framework contains the cases that the technique of \cite{La73} appearing in \cite[Theorem 7.3]{Ha85} and \cite[Proposition 3.1]{SLK04} is not  applicable.   
By (A1) - (A3), 
$$b_i - c_i \le 1- a_i \le 1 - b_{i}c_{i}.$$
Hence, 
\[ (1-a_i)^2 + 4b_i c_i \ge \min\{(b_i + c_i)^2, (1 + b_{i}c_{i})^2\} \ge 0. \]

Let 
\[ \alpha := \min\left\{0, \frac{c_0}{1 - a_0},  \frac{a_i - 1 + \sqrt{(1-a_i)^2 + 4b_i c_i}}{2b_i}  \ \middle| \ 1 \le i \le N-1  \right\}. \]
If $a_0 = 1$, then, we replace $c_0/(1 - a_0)$  in the right hand side of the above definition with $-1$. 

Let  
\[ \beta := \max\left\{0, \frac{c_0}{1 - a_0}, \frac{a_i - 1 + \sqrt{(1-a_i)^2 + 4b_i c_i}}{2b_i}  \ \middle| \ 1 \le i \le N-1 \right\}. \]
If $a_0 = 1$, then, we replace $c_0/(1 - a_0)$ in the right hand side of the above definition with $+\infty$.   

Let $Y := [\alpha, \beta]$. 
We consider the topology of $Y$ defined by the Euclid metric. 
For $k \in \Sigma_N$ and $y \in Y$, let 
\[ G_{k}(y) := \frac{(a_{k} - b_k c_k)(y+1)}{(b_{k}y + 1)\left((a_{k} + b_{k})y +c_{k} + 1\right)},  \textup{ and } H_k (y) :=  \frac{a_k y + c_k}{b_k y + 1}. \]  
If $a_0 = 1$, then, 
$$G_{0}(+\infty) := 1,  G_{k}(+\infty) := 0,  H_{0}(+\infty) :=  +\infty,  H_k (+\infty) := \frac{a_k}{b_k},  1 \le k \le N-1.$$  

\begin{Lem}\label{lem-fund-dR}
(i-1) If $a_0 = 1$, then, $\alpha = -1$ and $\beta = +\infty$. \\
(i-2) If $a_0 < 1$ and $b_{N-1} + c_{N-1} = 0$, then, $-1 = \alpha \le \beta < +\infty$.  \\
(i-3) If $a_0 < 1$ and $b_{N-1} + c_{N-1} > 0$, then, $-1 < \alpha \le \beta < +\infty$.\\
(ii) $H_i (y) \in Y$ for $i \in \Sigma_N, y \in Y$.\\  
(iii) \eqref{G-sum-1} holds.
\end{Lem} 

\begin{proof}
(i) By (A1), 
it holds that if $1 \le i \le N-2$,  then, 
\begin{equation}\label{minus-1-sand-1} 
\frac{a_i - 1 - \sqrt{(1-a_i)^2 + 4b_i c_i}}{2b_i} < -1 < \frac{a_i - 1 + \sqrt{(1-a_i)^2 + 4b_i c_i}}{2b_i}. 
\end{equation}
By (A2) and (A3), $b_{N-1} + c_{N-1} \ge 0$. 
By this and (A1), 
it holds that if $i = N-1$, then, 
\begin{equation}\label{minus-1-sand-2} 
\frac{a_{N-1} - 1 - \sqrt{(1-a_{N-1})^2 + 4b_{N-1} c_{N-1}}}{2b_{N-1}} = -1,  
\end{equation} 
and, 
\begin{equation}\label{minus-1-sand-3} 
-1 \le \frac{c_{N-1}}{b_{N-1}} =  \frac{a_{N-1} - 1 + \sqrt{(1-a_{N-1})^2 + 4b_{N-1} c_{N-1}}}{2b_{N-1}}. 
\end{equation} 

Hence, if $a_0 = 1$, then, $\alpha = -1$. 
If $a_0 < 1$, then, by (A1), $a_0 < c_0 + 1$. 
Hence, if $a_0 < 1$ and $b_{N-1} + c_{N-1} > 0$, then, $\alpha > -1$.

(ii) Let $i = 0$.   
First we remark that $H_0 (z) = a_0 z + c_0$. 
Assume $a_0 < 1$.
Then, by $\alpha \le c_0 / (1 - a_0) \le \beta$, 
we see that $\alpha \le H_0 (\alpha) \le H_0 (\beta) \le \beta$.
If $a_0 = 1$, then, $\beta = +\infty$. 
It is easy to see that $\alpha \le H_0 (\alpha)$.

Let $1 \le i \le N-1$.   
Then, $H_i (z) \ge z$ if and only if 
\begin{equation}\label{H-i} 
\frac{a_i - 1 - \sqrt{(1-a_i)^2 + 4b_i c_i}}{2b_i} \le z \le \frac{a_i - 1 + \sqrt{(1-a_i)^2 + 4b_i c_i}}{2b_i}. 
\end{equation} 
Hence, $H_i (\beta) \le \beta$.

By (A2), for every $k$, 
$H_k (z)$ is monotone increasing on $z \ge -1$. 
By this, \eqref{minus-1-sand-1}, \eqref{minus-1-sand-2}, \eqref{minus-1-sand-3}, and \eqref{H-i},   
$H_i (\alpha) \ge \alpha$.

(iii) It is easy to see by calculations that for every $i \ge 0$,  
\begin{equation*}
G_i (y) = \frac{(y+1)(b_{i+1} - b_i)}{(b_{i+1}y + 1)(b_i y + 1)}. 
\end{equation*} 
The assertion follows from this and (A1). 
\end{proof}

\begin{Lem}
If \textup{(sA3)} holds, then, \eqref{A-y} holds for $y = 0$. 
\end{Lem}

\begin{proof} 
By (sA3), $1 > a_0$ and $b_{N-1} + c_{N-1} > 0$. 
By this and Lemma \ref{lem-fund-dR} (i), 
we have that $-1 < \alpha \le \beta < +\infty$.
Therefore, \[ 0 < G_0 (\alpha) \le G_0 (\beta) < 1.\] 

Let $i \ge 1$. 
Since $\alpha > -1$ and $0 < b_i < 1$, 
\[ \inf_{y \ge \alpha} G_i (y) > 0.\] 
By the definition of $G_i$, 
we can show that 
if $y \ge -1$, then, $G_i (y) < 1$. 
Hence, by continuity of $G_i$, 
\[ \sup_{y \in [\alpha, \beta]} G_i (y) < 1.\]     
\end{proof} 

Hereafter we denote the set of fixed points of a map $f$ by $\textup{Fix}(f)$.  

\begin{Rem} 
(i) In the above proof, we have used $b_{N-1} + c_{N-1} > 0$. 
However, if $i < N-1$, then, $b_i + c_i > 0$ may fail. \\
(ii) If $i \ge 1$, $G_i$ may not be increasing. \\
(iii) \[ \left|\bigcap_{i \in \Sigma_N} \textup{Fix}(H_i)\right| \le  \left|\textup{Fix}(H_0)\right| = 1.\]
(iv) If $a_0 < 1$, then, $\textup{Fix}(H_0) = \left\{c_0/(1-a_0) \right\}$.
If $a_0 = 1$, then, by extending the domain of $H_0$, $\textup{Fix}(H_0) = \left\{+\infty \right\}$.  
\end{Rem} 

\begin{Lem} 
For the above definitions of $Y$, $\{G_i\}_i$ and $\{H_i\}_i$, assume that $\{\nu_y\}_{y \in Y}$ satisfies \eqref{g-me} for $Z = (\Sigma_N)^{\mathbb{N}}$ and $F_i (x) = ix$. 
Then, 
\[ \pi(x) = \sum_{i \ge 1} \frac{X_i (x)}{N^i}, \]
and, 
\[ \mu_{\varphi} = \nu_0 \circ \pi^{-1}. \]    
\end{Lem} 

By this lemma, we identify $((\Sigma_N)^{\mathbb{N}}, \{F_i\}_i)$ with $([0,1], \{S_i\}_i)$ where we let $S_i (z) := (z+i)/N$.  ; 

\begin{proof} 
Let \[ \begin{pmatrix} p_{n}(x) & q_{n}(x)  \\ r_{n}(x) & s_{n}(x) \end{pmatrix} := A_{X_{1}(x)} \cdots A_{X_{n}(x)}, \ \ n \ge 1. \] 
We have that  
\begin{align*}
\mu_{\varphi}\left(\pi (I_{k}(x)) \right) &= \mu_{\varphi}\left(\left[\sum_{i=1}^{k} \frac{X_{j}(x)}{2^{j}}, \sum_{i=1}^{k} \frac{X_{j}(x)}{2^{j}} + \frac{1}{2^{k}}\right)\right)\\
&= \Phi\left(A_{X_{1}(x)} \cdots A_{X_{k}(x)}; 1\right) - \Phi\left(A_{X_{1}(x)} \cdots A_{X_{k}(x)}; 0\right) \\
&= \frac{p_{k}(x)s_{k}(x)-q_{k}(x)r_{k}(x)}{s_{k}(x)(r_{k}(x)+s_{k}(x))}.
\end{align*}

Therefore by computation, (Recall \eqref{widetilde} for the definition of $\widetilde R_{0, n+1}(x)$.)
\[ \frac{\mu_{\varphi}\left(\pi (I_{n+1}(x)) \right) }{\mu_{\varphi}\left(\pi (I_{n}(x)) \right) } 
= \frac{\widetilde R_{0, n+1}(x)}{\widetilde R_{0,n}(x)} = \frac{\nu_0 (I_{n+1}(x))}{\nu_0 (I_{n}(x))}. \]

Hence $\mu_{\varphi}\left(\pi (I_{n}(x)) \right) = \nu_0 (I_{n}(x))$.
Since $\varphi$ is continuous, we have that 
\[ \lim_{n \to \infty} \mu_{\varphi}\left(\pi (I_{n}(x)) \right) = 0,\] 
and hence,  $\nu_0$ has no atoms. 
Since $\pi^{-1} (\pi (I_{n}(x))) \setminus I_{n}(x)$ is at most countable, 
\[ \nu_0 \left(\pi^{-1} (\pi (I_{n}(x))) \setminus I_{n}(x)\right) = 0.\] 
Since $\{\pi(I_n (x))  \mid  n \ge 1, x \in (\Sigma_N)^{\mathbb{N}} \}$ generates the Borel $\sigma$-algebra on $[0,1]$, 
\[ \nu_0 \circ \pi^{-1} = \mu_{\varphi}.\]  
\end{proof}

The following is the main result of this subsection. 

\begin{Thm}[Upper bound for Hausdorff dimension]\label{main-dR}
(i) If either 
\begin{equation}\label{G-non-empty} 
\bigcap_{i \in \Sigma_N}  G_i^{-1} \left(\frac{1}{N}\right) \ne \emptyset 
\end{equation}  
or 
\begin{equation}\label{include} 
\bigcap_{i \in \Sigma_N}  G_i^{-1} \left(\frac{1}{N}\right) \subset \bigcap_{i \in \Sigma_N} \textup{Fix}(H_i) 
\end{equation} 
fails, then, 
\[ \dim_H \mu_{\varphi} < 1. \]
(ii) If \eqref{G-non-empty} and \eqref{include} hold, 
then, 
the distribution function of $\mu_{\varphi}$ is given by 
\begin{equation}\label{ac} 
\varphi(x) = \frac{(N-1)x}{N -1 - c_{0} N (x-1)}, 
\end{equation}  
\end{Thm} 

In other words, if the solution $\varphi$ of \eqref{dR-N} is not of the form of \eqref{ac}, then, $\dim_H \mu_{\varphi} < 1$, 
and hence, $\varphi$ is a singular function.    
This is a special case of Theorem \ref{main-dR-gen} below. 
For proof, see the proof of Theorem \ref{main-dR-gen}.  

\begin{Rem}[Lower bound for Hausdorff dimension]
Assume $a_0 < 1$ and $b_{N-1} + c_{N-1} > 0$.  
Let $\widetilde c$ be the constant in \eqref{inf-multi}.   
Then, by Proposition \ref{transfer} and Lemma \ref{any}, 
\[ \dim_H \mu_{\varphi} \ge 1 + \frac{\inf\left\{s_N \left((p_j)_{j \in \Sigma_N} \right) \ \middle| \ (p_j)_{j \in \Sigma_N} \in P_N, \widetilde c \le p_j \le 1 - \widetilde c, \forall j \right\}}{\log N} > 0. \]
\end{Rem} 

\begin{Prop}[Regularity and singularity with respect to self-similar measures]\label{singu-ber-dR}
Let $p_i \in (0,1)$, $0 \le i \le N-1$.  
Let $e_0 := c_0 / (1-a_0)$ and assume $a_0 < 1$.     
Then,\\
(i) If there exist $p_i \in (0,1)$, $0 \le i \le N-1$ such that 
\begin{equation}\label{ai-ber} 
a_i = \frac{p_i + e_0 \sum_{j = 0}^{i} p_j}{\left(1 - \sum_{j = 0}^{i-1} p_j \right)e_0 + 1}.  
\end{equation} 
\begin{equation}\label{bi-ber}  
b_{i} = \frac{\sum_{j = 0}^{i-1} p_j}{\left(1 - \sum_{j = 0}^{i-1} p_j \right)e_0 + 1}. 
\end{equation}  
\begin{equation}\label{ci-ber}  
c_i = \frac{e_0 \left(\left(1- \sum_{j = 0}^{i} p_j \right) e_0 + 1 - p_i \right)  }{\left(1 - \sum_{j = 0}^{i-1} p_j \right) e_0 + 1}.  
\end{equation} 
hold for every $i$, 
then, $\mu_{\varphi}$ is absolutely continuous with $(p_0, \dots, p_{N-1})$-self-similar measure $\mu_{(p_0, \dots, p_{N-1})}$.\\
(ii) If the assumptions of (i) fails, then, $\mu_{\varphi}$ is singular with respect to every $(p_0, \dots, p_{N-1})$-self-similar measure. 
\end{Prop}

We do not know about an explicit expression for the Radon-Nikodym derivative $\dfrac{d\mu_{\varphi}}{ d\mu_{(p_0, \dots, p_{N-1})}}$. 

\begin{proof} 
(i) We first remark that in this case, $0 < a_0 = p_0 < 1$ and $\beta < +\infty$. 
By computation, 
\[ H_i ( e_0 ) = e_0,  \textup{ and  }  H_i^{\prime} ( e_0 ) = p_i,\]  
where $H_i^{\prime}$ denotes the derivative of $H_i$. 

Hence, each $H_i$ is contractive on a neighborhood $U$ of $e_0$.  
Since $H_0$ is contractive on $[\alpha, \beta]$, 
there exists $M$ such that it holds that $H_0^M (x) \in U$ for every $x \in [\alpha, \beta]$.  

By Azuma's inequality, 
\[ \nu_0 \left( \bigcup_{i \ge 1} \bigcap_{j = 1}^{M} \{X_{i+j} = 0\} \right) = 1.  \] 
Hence,  
\[ \lim_{n \to \infty} H_{X_n(x)} \circ \cdots \circ H_{X_1(x)}(0) = e_0,  \textup{ $\nu_0$-a.s.$x$.}\]   
and this convergence is exponentially fast\footnote{the speed of convergence depends on the choice of $x$.}. 
Hence,  
\[ \sum_{n \ge 1} \left(1 - \sum_{i \in \Sigma_N} \sqrt{p_i G_{i}\left( H_{X_n(x)} \circ \cdots \circ H_{X_1(x)}(0) \right)}\right) < +\infty,  \textup{ $\nu_0$-a.s.$x$.}\]
By this and \cite[Theorem VII.6.4]{Sh80}, 
$\nu_0$ is absolutely continuous with respect to $\nu_{(p_0, \dots, p_{N-1})}$.  
Hence, 
$\mu_{\varphi}$ is  absolutely continuous with respect to $\mu_{(p_0, \dots, p_{N-1})}$.  

(ii) Let $\mu_{(p_0, \dots, p_{N-1})}$ be  $(p_0, \dots, p_{N-1})$-self-similar measure.  
By the definition of $\pi$, 
we have that $\pi^{-1}(\pi(A)) \setminus A$ is at most countable for every $A$. 
First we consider the case that $p_0 \ne a_0$.  
Then, $\textup{Fix}(H_0) = e_0 \ne G_0^{-1}(p_0)$.  
Therefore, if $G_0 (y) = p_0$, then,  $G_0 (H_0 (y)) \ne p_0$. 
Thus \eqref{multi-sep-2} holds for $l = 1$ and $i_1 = 0$.   

Assume that $p_0 = a_0$. 
Then, either 
(a) $G_i (e_0) = p_i, i \in \Sigma_N$, or 
(b) $H_i (e_0) = e_0, i \in \Sigma_N$,   
fails, because if both (a) and (b) hold, then, \eqref{ai-ber}, \eqref{bi-ber} and \eqref{ci-ber} follow. 

Assume that (a) fails. 
Then, by using that $p_0 = a_0$ and $G_0$ is strictly increasing, 
$\cap_{i} G_i^{-1}(p_i) = \emptyset$.   
Since each $G_i$ is continuous and $Y$ is compact, 
\eqref{multi-sep-2} holds for every $l$ and $i_1, \dots, i_l$. 
Assume that (b) fails.  
Then, $H_i (e_0) \ne e_0$ for some $i$, and \eqref{multi-sep-2} holds for $l = 1$ and $i_1 =  i$. 
\end{proof}

\begin{Exa}[Linear case]
If all $c_i$ are zero, that is, all $g_i$ are affine maps, then, $\alpha = \beta = 0$, and hence $Y = \{0\}$ and $G_i (0) = a_i$. 
Then, $\mu_{\varphi}$ is $(a_0, \dots, a_{N-1})$-self-similar measure. 
We have that 
\[ \dim_H \mu_{\varphi} = 1+ \frac{s_N \left( (a_0, \dots, a_{N-1}) \right)}{\log N}. \]
\end{Exa}

\begin{Exa}
Let $N = 2$. 
Consider \eqref{dR-N} for 
\[ A_0 = \begin{pmatrix} 1 & 0 \\ 1 & 1 \end{pmatrix}, \ A_1 = \begin{pmatrix} 0 & 1 \\ -1 & 2 \end{pmatrix},   \]
then, $g_0 (x) = x/(x+1)$ and $g_1 (x) = 1/(2-x)$. 
Then, (A1)-(A3) holds, and hence, a unique solution $\varphi$ of \eqref{dR-N} exists. 
$\varphi$ is the {\it inverse} function of Minkowski's question-mark function \cite{Mi1904}. 
But (sA3) fails. 
Then, $Y = [-1, +\infty]$, 
\[ G_0 (x) = \frac{x+1}{x+2}, \ H_0(x) = x+1, \ H_1(x) = -\frac{1}{x+2}, \] 
and, $\mu_{\varphi} = \mu_0$.  
In this case, \eqref{A-y} may fail, but \eqref{multi-sep} holds for {\it every} $i_1, i_2 \in \{0,1\}$.    
By Lemma \ref{key}, 
there is a constant $\epsilon_1 > 0$ such that for every $i \in \N$ and {\it every} $x \in \{0,1\}^{\N}$, 
\begin{equation}\label{inverse-Minkowski-singular}  
s_{2}(p_{i} (y; x)) + s_{2}(p_{i + 1} (y; x)) < -\epsilon_1. 
\end{equation}

Hence, by Theorem \ref{main} (i), 
we have that $\dim_H \mu_{\varphi} < 1$.
In this case, for $y = 0$, \eqref{wAy} fails. 
We are not sure whether $\dim_H \mu_{\varphi} > 0$.  
By Proposition \ref{singu-ss}, 
$\mu_{\varphi}$ is singular with respect to every $(p_0, p_1)$-self-similar measure. 
\eqref{inverse-Minkowski-singular} is stable with respect to small perturbations for $g_0$ and $g_1$. 
See \cite[Example 2.8 (ii)]{Ok16} for a way of  small perturbations for $g_0$ and $g_1$.  
It is known that $\mu_{\varphi}$ is regular (\cite{MT19+}). 
\end{Exa}

\begin{Rem}
If $\widetilde \mu$ is the measure on $[0,1]$ such that its distribution function is Minkowski's question-mark function, then, by \cite{Kin60} it is known that 
\begin{equation}\label{Kinney} 
\dim_H \widetilde \mu = \frac{\log 2}{2 \int_{[0,1]} \log(1+x) \widetilde \mu(dx)} > \frac{1}{2}. 
\end{equation} 
\end{Rem}

\begin{Rem}
(i) In the case that each $g_i$ is a linear fractional transformation, 
we can take $Y$ as a subset of the set of real numbers. 
However, if  some $g_i$ are not linear fractional transformations, 
then, we may not be able to take $Y$ as a subset of $\R$. 
Therefore, our approach may not work well, at least in a direct manner.   
One difficulty is that a set of functions can be very large, informally speaking.\\
(ii) The approaches in \cite{Ha85, SLK04, Ok16} are different from the one used here. 
As an outline level, they are somewhat similar to each other. 
\end{Rem}

\subsection{Other homogeneous examples}

\begin{Exa}
We give an example that \eqref{multi-sep} fails for every $i_1, \dots, i_l$, $l = 1,2$ both,  
but \eqref{multi-sep} holds for every $i_1, \dots, i_l$, $l = 3$. 
Let $N = 2$. 
Let $Y = \R$. 
Let 
\[ G_0 (x) := \frac{1}{6}1_{\{x < 0, x > 1\}} + \left(x+\frac{1}{6}\right) 1_{\{0 \le x \le 1/2\}} + \left(\frac{7}{6} -x\right) 1_{\{1/2 \le x \le 1\}}.  \]
Let 
\[ H_0 (y) = H_1 (y) = \frac{5 - 3y}{6}. \]
Then, 
\[ G_0 \left(\frac{1}{3}\right) = G_0 \left(H_0 \left(\frac{1}{3}\right)\right) = G_0 \left(\frac{2}{3}\right) = \frac{1}{2}. \]
\[ G_0 \left(H_0^2 \left(\frac{1}{3}\right)\right) = G_0 \left(\frac{1}{2}\right) = \frac{2}{3}. \]  
\[ G_0 \left(H_0^2 \left(\frac{2}{3}\right)\right) = G_0 \left(H_0 \left(\frac{1}{2}\right)\right) = G_0 \left(\frac{7}{12}\right) = \frac{7}{12}. \]
\end{Exa} 

\begin{Exa}
Let $N = 2$ and $Y = Z = [0,1]$. 
Let 
\[ G_0 (x) = H_0 (x) = \frac{1}{4} 1_{\{x < 1/4\}} + x1_{\{1/4 \le x \le 3/4\}} + \frac{3}{4} 1_{\{x > 3/4\}}.  \]
Let $ H_1(x) = 1 - G_0(x)$.
Then, 
\[ G_0 \left(\frac{1}{2}\right) = G_1 \left(\frac{1}{2}\right) = H_0 \left(\frac{1}{2}\right) = H_1 \left(\frac{1}{2}\right) = \frac{1}{2}, \]
and for every $y \ne 1/2$,  $\dim_H \mu_y  < 1$.   
Therefore, if we did not introduce $Y(y)$ in the condition that 
\eqref{multi-sep} holds for some $\epsilon_0 \in (0, 1/N)$, $l \ge 1$ and $i_1, \dots, i_{l} \in  \Sigma_N$, 
 and furthermore simply assumed that the intersection of $\bigcap_{i \in \Sigma_N} G_{i}^{-1}\left( \left[ \frac{1}{N} - \epsilon_0, \frac{1}{N} + \epsilon_0\right]\right)$ and $\bigcap_{j \in \Sigma_N} \left(G_{j} \circ H_{i_{l}} \circ \cdots \circ H_{i_{1}}\right)^{-1} \left( \left[ \frac{1}{N} - \epsilon_0, \frac{1}{N} + \epsilon_0\right]\right)$  is empty,   
then, the converse of Theorem \ref{main} (i) would not hold. 
\end{Exa} 

\begin{Exa}
Fix $p \in (0,1)$. 
Let $N = 2$ and 
\[ Y = \left[-\min\{p,1-p\}^2, \min\{p,1-p\}^2\right] \] and $Z = [0,1]$.  
Let \[ G_0 (y) = p + \sqrt{|y|}, \textup{ and } H_0 (y) = H_1 (y) = \frac{|y|}{|y| + 1}. \]
Then, for every $x \in (\Sigma_N)^{\N}$ and $y \in Y$, 
\[ \lim_{n \to \infty} G_0 \circ H_{X_n (x)} \circ \cdots \circ H_{X_1 (x)}(y)  = p. \]
By this and Proposition \ref{transfer}, 
\[ \lim_{n \to \infty} \frac{-\log R_{y,n}(x)}{n} = \lim_{n \to \infty} s_2 \left( (G_j \circ H_{X_n (x)} \circ \cdots \circ H_{X_1 (x)}(y))_j \right) = s_2 (p,1-p), \ \textup{ $\nu_y$-a.s.$x$.} \]  
By using Example \ref{exa-cover} (i) and Lemmas \ref{exist} and \ref{any},        
we can show that for every $y \in Y$, $$\dim_H \mu_y = 1 + \frac{s_2 (p, 1-p)}{\log 2}.$$

Furthermore,  $\mu_y$ is the product measure on $\{0,1\}^{\N}$ such that 
\[ \mu_y (X_n = 0) =  p + \sqrt{H_0^{n-1} (y)} = p + \sqrt{\frac{|y|}{(n-1)|y| + 1}}.  \]

If $y = 0$, then, $\mu_y$ is the $(p, 1-p)$-self-similar measure.  
If $y \ne 0$, then, by Kakutani's dichotomy \cite{Kak48},  
$\mu_y$ is singular with respect to the $(p, 1-p)$-self-similar measure.  
\end{Exa} 

\begin{Exa}
Let $Z = [0,1]$ and $N = 2$.  
Let $Y = \R$.  
Let \[ G_0 (y) := \max\left\{\frac{1}{2} - |y|, 0 \right\},  H_0 (y) := (1-\epsilon)y, \textup{ and, }  H_1 (y) := \frac{y}{\epsilon}.  \] 
Let $y \ne 0$. 
Then, \eqref{A-y} fails. 
By using that $H_0$ is contractive and $G_0 (0) = G_1(0) = 1/2$ and $0$ is the fixed points of $H_0$ and $H_1$ both,  
we see that it does not hold that \eqref{multi-sep} holds for some $\epsilon_0 \in (0, 1/N)$, $l \ge 1$ and $i_1, \dots, i_{l} \in  \Sigma_N$. 

We will show that $\dim_H \mu_{y} < 1$. 
For simplicity we assume $y = 1/4$.  
By Azuma's inequality, if we take sufficiently small $\epsilon > 0$, 
then, 
there exists $D > 1$ such that for every $n \ge 1$, 
\[ \ell \left(\left\{x \in \{0,1\}^{\N} \ \middle| \ \left| H_{X_n (x)} \circ \cdots \circ H_{X_1 (x)} \left(\frac{1}{4}\right) \right| < \frac{1}{2} \right\}\right) \le D^{-n}, \] 
where $\ell$ is the $(1/2, 1/2)$-Bernoulli measure on $\{0,1\}^{\N}$. 
Therefore, by the definition of $\nu_y$, for each $n$, 
\[ \left| \left\{(i_1, \dots, i_n) \in \{0,1\}^n \ \middle| \  \nu_y (I(i_1, \dots, i_n, 0)) > 0 \right\} \right| \] 
\[ =  \left| \left\{(i_1, \dots, i_n) \in \{0,1\}^n \ \middle| \ \left| H_{i_n} \circ \cdots \circ H_{i_1} \left(\frac{1}{4}\right) \right| <  \frac{1}{2} \right\} \right| \le \left(\frac{2}{D}\right)^n.\]
Hence, by taking sum with respect to $n$,   
\[ \left| \{(i_1, \dots, i_m) \in \{0,1\}^m \ \middle| \  \nu_y (I(i_1, \dots, i_m)) > 0 \} \right| \] 
\[ = \sum_{n < m} \left| \{(i_1, \dots, i_n) \in \{0,1\}^n \ \middle| \  \nu_y \left(I(i_1, \dots, i_n, 0, 1, \dots, 1)\right) > 0 \} \right| \]
\[\le \sum_{n < m} \left(\frac{2}{D}\right)^n \le C \left(\frac{2}{D}\right)^{m}. \]
By using this and recalling that $D > 1$, $\dim_H \mu_{1/4} < 1$.
\end{Exa} 

\begin{Exa}
Let $Z = [0,1]$ and $N = 2$ and $S_i (z) = (z+i)/2$. 
It is easy to construct an example such that $\mu_y$ is not absolutely continuous or singular with respect to $(p,1-p)$-self-similar measure $\mu_{(p,1-p)}$. 
Let $Y = \{0,1,2\}$. 
Let \[ H_0 (0) = 1, H_1(0) = 2, \textup{ and } H_i (j) = j, \ j = 1,2, \ i = 0,1.\]  
Let \[ G_0 (j) = \frac{1}{2}, \ j = 1,2.\]   
Let $\mu_1 = \mu_{(p,1-p)}$ and $\mu_2$ be a probability measure which is singular with respect to $\mu_{(p,1-p)}$. 
Then, by the uniqueness of the Lebesgue decomposition, $\mu_0$ is not absolutely continuous or singular with respect to $\mu_{(p,1-p)}$.  
\end{Exa}

\subsection{Inhomogeneous cases}

We finally deal with the case that $Z$ is an inhomogeneous self-similar set.  This has an application of a further generalization of \eqref{dR-N}.

Let $g_0 (x) := p_0 x$ and 
\[ g_i (x) := p_i x  + (p_0 + \cdots + p_{i-1}), i \ge 1. \]
Consider 
\begin{equation}\label{dR-gen}
	\varphi(x) = \begin{cases}
						g_{0}\left(\varphi(S_0^{-1}(x))\right) & 0 \le x \le q_0, \\
						g_{1}\left(\varphi(S_1^{-1}(x))\right) & q_0  \le x \le q_0 + q_1, \\
						\cdots\\
						g_{N-1}\left(\varphi(S_{N-1}^{-1}(x))\right) & q_0 +\cdots + q_{N-2}  \le x \le 1,     
				          \end{cases}
\end{equation}
Then there exists a unique solution $\varphi$ of \eqref{dR-gen} and let $\mu = \mu_{\varphi}$ be the  probability measure whose density function is $\varphi$.    

We now consider the case that $\{g_i\}_i$ are linear fractional transformations on $[0,1]$ satisfying (A1) -(A3) in Subsubsection 3.4.2.  
Define a space $Y$ and  functions $G_i, H_i, i \in \Sigma_N$ on $Y$ in the same manner as in Subsubsection 3.4.2. 
Let $q_i, i \in \Sigma_N$ be positive numbers satisfying that $$q_0 + \cdots + q_{N-1} = 1.$$
Let $S_0 (x) := q_0 x$ and 
\[ S_i (x) := q_i x  + (q_0 + \cdots + q_{i-1}), i \ge 1. \]
Then we have the following generalization of Theorem \ref{main-dR}. 

\begin{Thm}\label{main-dR-gen}
(i) If either 
\begin{equation}\label{G-non-empty-gen} 
\bigcap_{i \in \Sigma_N}  G_i^{-1} \left( q_i\right) \ne \emptyset 
\end{equation}  
or 
\begin{equation}\label{include-gen} 
\bigcap_{i \in \Sigma_N}  G_i^{-1} \left(q_i\right) \subset \bigcap_{i \in \Sigma_N} \textup{Fix}(H_i) 
\end{equation} 
fails, then, 
\[ \dim_H \mu_{\varphi} < 1. \]
(ii) If \eqref{G-non-empty-gen} and \eqref{include-gen} hold, 
then, 
the distribution function $\varphi$ of $\mu_{\varphi}$ is given by 
\begin{equation}\label{ac-gen} 
\varphi(x) = \frac{(1-q_0)x}{1 -q_0 + c_0 - c_0 x}, \textup{  where we let $C_{0} := \frac{c_0}{1-q_0}$. }
\end{equation}  
\end{Thm}

\begin{proof}
(i) If \eqref{G-non-empty-gen} fails and $\beta < +\infty$, 
then, by the continuity of $G_i$ and the compactness of $Y$, 
we have that 
$$\inf_{y \in Y} \sum_{i \in \Sigma_N} \left| G_i (y) - q_i \right| > 0.$$
This holds even when $\beta = +\infty$, by recalling the definition of $G_i (+\infty)$. 
Hence, 
\[ \sup_{x \in (\Sigma_N)^{\mathbb{N}}, y \in Y, i \in \mathbb{N}} s_N \left(p_i (y; x)\right) < 0,  \]
where $s_N$ denotes the relative entropy with respect to $(q_i)_i$. 
Now the assertion follows from Proposition \ref{transfer} and Lemma \ref{exist}. 

Assume that  \eqref{G-non-empty-gen} holds and \eqref{include-gen} fails. 
Let 
\[ y_0 := G_0^{-1} \left(q_0\right). \]
Then, by \eqref{G-non-empty-gen}, we have that 
\begin{equation}\label{unique-gen} 
\bigcap_{i \in \Sigma_N}  G_i^{-1} \left(q_i\right)  = \left\{y_0\right\}. 
\end{equation} 
Since \eqref{include-gen} fails, we have that for some $i$, 
\[ H_i \left(y_0\right) \ne y_0. \]
Since $G_0$ is monotone increasing, 
\eqref{multi-sep} holds for $l = 1$ and $i_1 =  i$. 

(ii) By the assumption, \eqref{unique-gen} holds, and hence it holds that for each $i \in \Sigma_N$, 
\[ G_i \left( y_0 \right) = q_i, \textup{ and }   H_i \left( y_0 \right) = y_0.  \] 

By using them, 
we see that for every $i \in \Sigma_N$, 
\begin{equation}\label{a-i-c-i-gen} 
a_i = (1+q_i) y_0 b_i + q_i, \textup{ and } c_i = y_0 (1- q_i - q_i y_0 b_i). 
\end{equation}

By induction in $i$, we can show that 
\begin{equation}\label{b-i-gen}   
b_i = \frac{q_0 + \cdots + q_{i-1}}{1 + y_0 (1 - (q_0 + \cdots + q_{i-1}))},  \ i \ge 1. 
\end{equation} 

If $i = 0$, then, by (A1), $b_0 = 0$. 
We see that 
\[ y_0 = \frac{c_0}{1-q_0}. \]
By \eqref{a-i-c-i-gen} and \eqref{b-i-gen},  it is easy to see that $\varphi$ given by \eqref{ac-gen} satisfies \eqref{dR-gen}.  
\end{proof} 

In the special case that each $g_i$ is linear, we have the following: 
\begin{Cor}
Let $N \ge 2$. 
Let $p_i, q_i, i \in \Sigma_N$ be positive numbers satisfying that 
$$p_0 + \cdots + p_{N-1} = q_0 + \cdots + q_{N-1} = 1.$$
Let $S_0 (x) := q_0 x$ and 
\[ S_i (x) := q_i x  + (q_0 + \cdots + q_{i-1}), i \ge 1. \]
Let $g_0 (x) := p_0 x$ and 
\[ g_i (x) := p_i x  + (p_0 + \cdots + p_{i-1}), i \ge 1. \]
Consider \eqref{dR-gen}. 
If $(p_0, \cdots, p_{N-1}) \ne (q_0, \cdots, q_{N-1})$, then, 
\[ \dim_{H} \mu_{\varphi} < 1. \]
Otherwise, $\mu$ is the Lebesgue measure on $[0,1]$.  
\end{Cor}

In this case, \cite[Theorem 3.9]{Ok19+} could also give singularity of $\mu_{\varphi}$. However, it is not sufficient to give singularity  if $g_i$ is not linear for some $i$, for example, it is a linear fractional transformation as above.

\section{Open problems}

(1) Give estimates for the upper and lower local dimensions of $\mu_y$ and consider the Hausdorff dimensions for the level sets.    
Consider other notions of dimensions for $\mu_y$, such as $L^p$-dimensions. 
See \cite{St93}.  

(2) If $Z = [0,1]$, then, under what conditions is $\mu_y$ a Rajchman measure \cite{R28}, that is, a measure whose Fourier transform vanishes at infinity? 
See \cite{Ly95} for more information of Rajchman measures.   
Recently, \cite{JS16} shows that the Fourier coefficients of $\widetilde \mu$ decay to zero, 
by considering conditions which assure a given measure invariant with respect to the Gauss map is a Rajchman measure, and using \eqref{Kinney}. 

(3) Let $Z = [0,1]$. 
It is natural to consider structures of $L^2 ([0,1], \mu_y)$. 
For example, find a subset $P$ of $\N$ such that $\{\exp(2\pi i k) : k \in P\}$ forms an orthonormal basis of $L^2 ([0,1], \mu_y)$.   
\cite{JP98} considers this question for a class of self-similar measures.     

\vspace{1\baselineskip} 

{\it Acknowledgment.}    
The author appreciates a referee for suggesting an extension of our main results and many comments. 
The author was supported by JSPS KAKENHI 16J04213, 18H05830,  19K14549 and also by the Research Institute for Mathematical Sciences, a Joint Usage/Research Center located in Kyoto University.

\end{document}